\documentclass[a4paper,11pt]{article}
\usepackage{fullpage}
\usepackage{graphicx}
\usepackage[utf8]{inputenc}
\usepackage[english]{babel}

\usepackage{amsmath}
\usepackage{amssymb,amsthm,dsfont,enumerate,url,appendix,tabularx}
\usepackage{todonotes, comment}

\usepackage{psfrag}
\usepackage{graphicx}
\usepackage{tikz}
\usepackage{pgfplots}

\theoremstyle{plain}
\numberwithin{equation}{section}
\numberwithin{figure}{section}
\numberwithin{table}{section}

\newtheorem{theorem}{Theorem}[section]
\newtheorem{lemma}[theorem]{Lemma}
\newtheorem{corollary}[theorem]{Corollary}
\newtheorem{proposition}[theorem]{Proposition}

\newtheorem{assumption}[theorem]{Assumption}

\theoremstyle{definition}
\newtheorem{definition}[theorem]{Definition}
\newtheorem{example}[theorem]{Example}

\theoremstyle{remark}
\newtheorem{remark}[theorem]{Remark}

\newcommand{\R}{\mathbb{R}}

\newcommand{\N}{\mathbb{N}}

\newcommand{\TB}{\mathfrak T}
\newcommand{\fcc}{{\tt fcc}}

\newcommand{\abs}[1]{\left|#1\right|}
\newcommand{\norm}[1]{\left\|#1\right\|}
\newcommand{\average}[2]{{\mathcal A}_{#1} #2 }

\newcommand{\dist}[2]{\operatorname{dist}\left(#1,#2\right)}

\def\R{{\mathbb R}}
\def\BbbR{{\mathbb R}}
\def\N{{\mathbb N}}
\def\BbbN{{\mathbb N}}

\newcommand{\TT}{{\mathcal T}}
\newcommand{\Tt}{\widetilde {\mathcal T}}
\newcommand{\Th}{\widehat {\mathcal T}}

\newcommand{\II}{I_h^m}
\newcommand{\hh}{\widehat h}

\def\support#1{\operatorname*{supp}#1}
\def\diam{\operatorname*{diam}}

\newcommand{\eremk}{\hbox{}\hfill\rule{0.8ex}{0.8ex}}

\newcommand{\skp}[1]{\left< #1 \right>}
\begin{document}
\title{On the stability of Scott-Zhang type operators and application to multilevel preconditioning in fractional diffusion}
\author{
\and {Markus Faustmann\footnotemark[1]}
\and {Jens Markus Melenk\footnotemark[1]}
\and {Maryam Parvizi\footnotemark[1]}
}
	\footnotetext[1]{Institute of Analysis and Scientific Computing, 
	Technische Universit\"at Wien, Wiedner Hauptstr. 8-10, A--1040 Wien, 
        Austria, 
\texttt{$\{$markus.faustmann, melenk, maryam.parvizi$\}$@tuwien.ac.at} \\
MP was funded by the Austrian Science Fund (FWF) project P 28367 and JMM was supported 
by the Austrian Science Fund (FWF) by the special
research program Taming complexity in PDE systems (grant SFB F65)}
\maketitle
\begin{abstract}  
We provide an endpoint stability result for Scott-Zhang type operators 
in Besov spaces. For globally continuous piecewise polynomials
these are bounded from $H^{3/2}$ into $B^{3/2}_{2,\infty}$; 
for elementwise polynomials these are bounded from $H^{1/2}$ into 
$B^{1/2}_{2,\infty}$. 
As an application, we obtain a multilevel decomposition based on Scott-Zhang 
operators on a hierarchy of meshes generated by newest vertex bisection with 
equivalent norms up to (but excluding) the endpoint case. 
A local multilevel diagonal preconditioner for the fractional Laplacian on
locally refined meshes with optimal eigenvalue bounds is presented. 
\end{abstract}

\section{Introduction}

The Scott-Zhang projection, originally introduced in \cite{scott-zhang90},
is a very important tool in numerical analysis and has been generalized 
in various ways, 
\cite{bernardi,girault1,cc1,cchu,apel2,a01,rand12,ciarlet13, falk-winther15,AFFKP15, karkulik-melenk15, ern-guermond17}. 
In its classical form, it is quasi-local, it is a projection onto the space of globally continuous, piecewise polynomials, 
it is stable in both $L^2$ and $H^1$ (and thus, by interpolation also in $H^s$, $s \in (0,1)$), 
and has optimal approximation properties.  
Therefore, it is well-suited for the analysis of classical finite element methods (FEMs), \cite{brenner-scott02},
and plays a key role in the analyses of, e.g., anisotropic finite elements, \cite{Apel99}, adaptive 
finite element methods, \cite{AFKPP13}, or mixed methods, \cite{Badia12}. 


As globally continuous piecewise linear functions are not only in the Sobolev space $H^1(\Omega)$, but 
also in (fractional) Sobolev spaces $H^{3/2-\varepsilon}(\Omega)$ for any $\varepsilon >0$ --- in fact, they are
in the Besov space $B^{3/2}_{2,\infty}(\Omega)$ --- a natural question is whether 
the operator is also stable in the stronger norms imposed on these spaces.
In this article, we provide an endpoint stability result, 
i.e., study the stability in the norm $\|\cdot\|_{B^{3/2}_{2,\infty}}$,  
not only for the Scott-Zhang operator but more generally for local, $L^2(\Omega)$-stable operators 
with certain approximation properties in $L^2(\Omega)$ on shape-regular meshes. Additionally, we cover the case of 
operators such as the elementwise $L^2$-projection that map into spaces of discontinuous piecewise polynomials, 
where the corresponding endpoint space is $B^{1/2}_{2,\infty}$. By interpolation, these endpoint results
imply stability results in the full range between $L^2$ and the Besov space. 

Multilevel representations of Sobolev spaces (and Besov spaces) based on sequences of uniformly refined 
meshes are available in the literature; see, e.g.,  
\cite{oswald94,schneider98,bramble-pasciak-vassilevski99}, and the references there.  
Our stability for Scott-Zhang type operators allows us develop multilevel norm equivalences for 
spaces of globally continuous piecewise polynomials on 
adaptively refined meshes $\TT$. These are assumed to be shape-regular and obtained by 
\emph{newest vertex bisection} (NVB). The mesh hierarchy $\widetilde{\TT}_\ell = \fcc(\TT,\widehat{\TT}_\ell)$, 
$\ell=0,\ldots,L$, is given by the finest common coarsening 
of $\TT$ and the meshes $\widehat{\TT}_\ell$ 
of a sequence $(\widehat{\TT}_\ell)_{\ell}$ of uniformly refined NVB-generated meshes. 
Our actual multilevel decomposition is then obtained with an adapted Scott-Zhang operator that is of 
independent interest (Lemma~\ref{lemma:scott-zhang-tilde-hat}).

In numerics, an important application of multilevel decompositions is the design of 
multilevel additive Schwarz preconditioners, in particular multilevel diagonal scaling 
\cite{dryja-widlund91,zhang92} and BPX, \cite{bramble-pasciak-xu91}. 
In this article, we propose a local multilevel diagonal preconditioner for 
the integral fractional Laplacian $(-\Delta)^s$ for $s\in(0,1)$ on adaptively refined meshes $\TT_\ell$. 
The need for a preconditioner arises from the observation that the condition number 
 of the stiffness matrix $\mathbf{A}^\ell \in \R^{N_\ell \times N_\ell}$ corresponding to a FEM discretization
 by piecewise linears of the integral fractional Laplacian grows like
$
  \kappa(\mathbf{A}^\ell) \sim N_\ell^{2s/d} \left(\frac{h^\ell_{\rm max}}{h^\ell_{\rm min}}\right)^{d-2s},
$
 where $h^\ell_{\rm max},h^\ell_{\rm min}$ denote the maximal and minimal mesh width of $\mathcal{T}_\ell$, 
see, e.g., \cite{AMcLT99,AinGlu17}. 
Since the fractional Laplacian on bounded domains features singularities 
at the boundary, typical meshes are strongly refined towards the boundary so that 
the quotient $h^\ell_{\rm max}/h^\ell_{\rm min}$ is large 
(see, e.g., \cite{AinGlu17,BBNOS18,faustmann2019quasi} for adaptively generated meshes). While the impact 
of the variation of the element size can be controlled by diagonal scaling (see, e.g., \cite{AMcLT99}) 
the factor $N_\ell^{2s/d}$ persists. A good preconditioner is therefore required for an efficient iterative solution 
for large problem sizes $N_\ell$. 
Indeed, preconditioning for fractional differential operators has attracted attention recently. 
We mention multigrid preconditioners, \cite{AinGlu17} based on uniformly refined mesh hierarchies and 
operator preconditioning, \cite{hiptmair06,gimperlein-stocek-urzua19,stevenson-vanvenetie19}, which requires one
to realize an operator of the opposite order. Another, classical technique is the framework of 
additive Schwarz preconditioners, analyzed in a BPX-setting with Fourier techniques in \cite{BorNoc19}. 
In the present work, we also adopt the additive Schwarz framework and show that, also in the presence of 
adaptively refined meshes, multilevel diagonal scaling leads to uniformly bounded condition numbers
for the integral fractional Laplacian. The above mentioned norm equivalence of the multilevel decomposition
provides the lower bound for the eigenvalues; an inverse estimate in fractional Sobolev norms, 
similarly to \cite{faustmann2019quasi}, gives the upper bound for the eigenvalues. 
We mention that very closely related to preconditioning of discretizations of the fractional differential operators 
is earlier work on preconditioning for the hypersingular integral equation in boundary element methods (BEM), 
\cite{tran_stephan_asm_h_96,tran_stephan_mund_hierarchical_prec,
tran-stephan-zaprianov98,ainsworth2003multilevel,
maischak_multilevel_asm,feischl2017optimal}. 

The present paper is structured as follows: Section~\ref{sec:mainresults} provides the necessary 
notation and states the three main results of the paper. The first result is 
the stability of quasi-interpolation operators in the endpoint Besov space (Theorem~\ref{thm:stability-discontinuous}) 
both for globally continuous piecewise polynomials and discontinuous piecewise polynomials. 
The second result is a multilevel decomposition based on a 
modified Scott-Zhang operator on a mesh hierarchy of NVB meshes (Theorem~\ref{thm:multilevel-fcc}).
The third result is an optimal local multilevel 
diagonal preconditioner for the fractional Laplacian on adaptively generated meshes and 
the discretization 
by piecewise linears $S^{1,1}_0(\TT)$ (for $s \in (0,1)$) and by piecewise constants 
$S^{0,0}(\TT)$ (for $s \in (0,1/2)$). 
Two types 
of mesh hierarchies are considered: The first one is assumed to be generated by an adaptive algorithm 
and discussed in Theorem~\ref{thm:ASfractional}. The second one, 
$\widetilde{\TT}_\ell = \fcc(\TT,\widehat{\TT}_\ell)$, is generated 
by taking the finest common coarsening of a fixed mesh $\TT$ and a sequence of uniformly 
refined meshes $\widehat{\TT}_\ell$; this is analyzed in Theorem~\ref{thm:ASfractional2}. 

Section~\ref{sec:SZproof} is concerned with the proof of the stability result 
of the quasi-interpolation operators in Besov spaces. Moreover, we present some extensions  
such as inverse estimates in Besov-norms (Lemma~\ref{lemma:inverse-estimate}) or an interpolation result 
for discrete spaces in Besov-norms (Corollary~\ref{cor:interpolation}).

In Section~\ref{sec:NVBproof}, we develop properties of the finest common coarsening of two 
meshes. We prove the norm equivalence for the multilevel decomposition. 
Furthermore, we develop, for given meshes $\TT$, $\widehat \TT$, two Scott-Zhang type operator 
$\widehat{I}^{SZ}$ and 
$\widetilde{I}^{SZ}$ on the meshes $\widehat \TT$ and $\widetilde{\TT}:= \fcc(\TT,\widehat\TT)$ 
with the property $\widehat{I}^{SZ} u = \widetilde{I}^{SZ} u$ for $u \in S^{p,1}(\TT)$. Such 
operators are useful in various context and similar operators have been constructed, for example, 
in \cite[Lemma~{3.5}]{diening-kreuzer-stevenson16}. 

Finally, Section~\ref{sec:ASproof} provides the abstract analysis for the additive Schwarz method 
to prove the optimal bounds on the extremal eigenvalues of the preconditioned stiffness matrix 
for the fractional Laplacian on adaptively generated NVB meshes. Numerical experiments underline 
the optimality of the preconditioner.

Throughout the paper, we use the notation $\lesssim$ to abbreviate $\leq$ up to 
a generic constant $C>0$ that does not depend on critical parameters in our analysis. Moreover, 
we use $\simeq$ to indicate that both estimates $\lesssim$ and $\gtrsim$ hold.

\section{Main results}
\label{sec:mainresults}

\subsection{Stability of (quasi-)interpolation operators in Besov spaces}
Let $\Omega \subset \R^d$ be a bounded Lipschitz domain.
For $s \geq 0$, we use the Sobolev spaces  $H^s(\Omega)$, 
in the integer case $s \in \N_0$ defined in the standard way, see, e.g., \cite{adams-fournier03},
and for the fractional case $s \notin \N_0$ defined by interpolation, \cite{tartar07}.
We note that, equipped with the Aronstein-Slobodeckij (semi-)norm
\begin{align*}
\norm{u}_ {H^s (\Omega)}^2:= \norm{u}_{L^2 (\Omega)}^2+\abs{u}_{H^s (\Omega)}^2 \qquad \text{with} \quad 
\abs{u}_ {H^s (\Omega)}^2 
:=\int_\Omega \int_{\Omega} \frac{(u(x)-u (y))^2}{|x-y|^{d+2s}},
\end{align*}
the space $H^s (\Omega)$ is a Hilbert space.

Moreover, for $s > 0$, $s \notin \N_0$, $q \in [1,\infty]$, we employ the Besov spaces $B^{s}_{2,q}(\Omega)$ 
defined as the interpolation spaces 
$B^{s}_{2,q}(\Omega) 
:= (H^\sigma(\Omega),H^{\sigma+1}(\Omega))_{\theta,q}$, where
$\sigma = \lfloor s\rfloor$ and $\theta = s - \sigma \in (0,1)$. The 
norm is given by 
\begin{align*}
\|u\| _{B_{2,q} ^s(\Omega)}:= 
\begin{cases}
\left(\int_{t=0}^\infty \left(t^{-\theta} K(t,u)\right)^q\frac{dt}{t}\right)^{1/q} & 
q \in [1,\infty), \\
\sup _ {t>0} t ^{-\theta} K(t,u) & q = \infty. 
\end{cases}
\end{align*}
Here, for $u \in H^{\sigma}(\Omega)$ and  $t>0$, the $K$-functional is defined by
	\begin{align*}
 K(t,u):=	\inf _{u_t \in H^  {\sigma+1} (\Omega)} \|u-u_t\| _{H^ \sigma (\Omega)}+t \|u_t\| _{H^ {\sigma+1}(\Omega)}.
	\end{align*}
	
For discretization, we assume that a regular (in the sense of Ciarlet) triangulation $\TT$ of $\Omega$ consisting 
of open simplices is given. Additionally, $\TT$ is assumed to be $\gamma$-shape regular in the sense
\begin{align*}
 \max_{T \in \TT} (\operatorname*{diam}(T)/\abs{T}^{1/d}) \leq \gamma < \infty,
\end{align*}
where $\operatorname*{diam}(T):= \sup_{x,y \in T} \abs{x-y}$ and $\abs{T}$ is the volume of $T$.
By $h \in L^\infty(\Omega)$, we denote the piecewise constant mesh size function satisfying
$h|_T:=h_T:= \abs{T}^{1/d}$ for $T \in \TT$. 

Let $P_p (T )$ be the space of polynomials of (maximal) degree $p$ on the element $T\in\TT$.
Then, the spaces of $\TT$-piecewise polynomials of degree $p \in \N_0$ and regularity $m \in \N_0$
are defined by 
\begin{align*}
S ^{p,m} (\mathcal{T})&:= \left\lbrace u \in H^m (\Omega) 
\colon u|_T \in P_p (T)\quad \forall T \in \mathcal{ T} \right\rbrace, \\  
S ^{p,m}_0 (\mathcal{T})&:= S ^{p,m} (\mathcal{T})\cap H^1_0(\Omega) \qquad m\geq 0.  
\end{align*}
For $T  \in \mathcal{T} $ the element patch 
\begin{equation*}
\omega (T):=\omega^1 (T):= \operatorname*{interior}  \left(\bigcup \{\overline{T^\prime}\colon T' \in\TT, \overline{T^\prime} \cap \overline{T} \ne \emptyset\}\right),
\end{equation*}
consists of the element $T$ and all its neighbors. Inductively, the $k$-th order patch is defined by 
\begin{align*}
 \omega^k (T) := \operatorname*{interior}  \left(\bigcup \{\overline{T^\prime}\colon T' \in\TT,
 \overline{T^\prime} \cap \overline{\omega^{k-1}(T)} \ne \emptyset\}\right).
\end{align*}

In the following, we study (quasi-)interpolation operators 
$\II$ satisfying the following locality, stability and approximation properties.

\begin{assumption}
\label{assumption:I-discontinuous}
Let $m\geq 1$ and $\II$ be an operator $\II:L^2(\Omega) \rightarrow S^{p,m-1}(\TT) $ that satisfies:
\begin{enumerate}[(i)]
\item 
\label{item:assumption:I-discontinuous-i}
Quasi-locality: For every $T \in \TT$ the restriction
$(\II u)|_T$ depends solely on $u|_{\omega (T)}$. 
\item 
\label{item:assumption:I-discontinuous-ii}
Stability in $L^2$: For $u \in L^2(\Omega)$, there holds
$$
\|\II u \|_{L^2(T)} \leq C \|u\|_{L^2(\omega (T))}. 
$$
\item 
\label{item:assumption:I-discontinuous-iii}
Approximation properties of $m$-th order: For $u \in H^m(\Omega)$, there holds
$$
\|u - \II u\|_{L^2(T)} \leq C h_T^m \|u\|_{H^m(\omega (T))}. 
$$
\end{enumerate}
The constants in \eqref{item:assumption:I-discontinuous-ii} and \eqref{item:assumption:I-discontinuous-iii}
depend only on $\Omega$, $d$, $m$, $p$, and the $\gamma$-shape regularity of $\TT$. 
\end{assumption}

The following theorem is the main result of this subsection and states a stability result in the 
Besov space $B^{m-1/2}_{2,\infty}(\Omega)$
for operators satisfying Assumption~\ref{assumption:I-discontinuous}. Its proof will be given in 
Section~\ref{sec:proofSZ} below.

\begin{theorem}
\label{thm:stability-discontinuous}
Fix $m \in \{1,2\}$ and $p \in \BbbN_0$ with $p \ge m-1$. 
Let an operator $\II$ satisfying Assumption~\ref{assumption:I-discontinuous} be given. 
Then, for all $u \in H^{m-1/2}(\Omega)$, we have
\begin{equation}
\label{eq:thm:stability-discontinuous-i}
\|\II u\|_{B^{m-1/2}_{2,\infty}(\Omega)} \leq C \|u\|_{H^{m-1/2}(\Omega)}, 
\end{equation}
where the constant $C > 0$ depends solely on $\Omega$, $d$, $m$ ,$p$, and the $\gamma$-shape regularity of ${\mathcal T}$. 

If the mesh ${\mathcal T}$ is additionally {\em quasi-uniform}, then, for all $u \in B^{m-1/2}_{2,\infty}(\Omega)$, 
the sharper estimate 
\begin{equation}
\label{eq:thm:stability-discontinuous-ii}
\|\II   u\|_{B^{m-1/2}_{2,\infty}(\Omega)} \leq C \|u\|_{B^{m-1/2}_{2,\infty}(\Omega)}  
\end{equation}
holds.
\end{theorem}

\begin{remark}
For $m=1$, a possible choice for $\II$ is the $L^2(\Omega)$-orthogonal projection that trivially satisfies 
Assumption~\ref{assumption:I-discontinuous}. 
For $m=2$ the Scott-Zhang projection, introduced in 
\cite{scott-zhang90} and defined below, is an example of an operator $\II$
satisfying Assumption~\ref{assumption:I-discontinuous}. 
Therefore, Theorem~\ref{thm:stability-discontinuous}
provides a novel stability estimates for these projection operators in Besov spaces.
\eremk
\end{remark}

\subsection{Multilevel decomposition based on mesh hierarchies generated by NVB}
The multilevel decompositions will be based on mesh hierarchies that are engendered by 
\emph{newest vertex bisection} (NVB). For discussion of properties of NVB meshes we refer to \cite{KPP13} for 
the case $d = 2$ and to \cite{stevenson08} for the case $d \ge 3$. 
We consider sequences of regular meshes that are obtained by NVB-refinement from an
initial mesh $\TT_0$. 

\subsubsection{The finest common coarsening}

For two regular triangulations $\TT$, $\TT^\prime$ (obtained by NVB from
the same triangulation $\widehat\TT_0$), we define the \emph{finest common coarsening} as 
\begin{align}
\label{eq:def-fcc}
& \fcc(\TT,\TT^\prime) := \\
\nonumber 
& \qquad \underbrace{ \{T \in \TT\,:\, \exists T' \in \TT^\prime \mbox{ s.t. }  T' \subsetneq T\}}_{=:\TB_1}
\cup \underbrace{ \{T^\prime \in \TT^\prime\,:\, \exists T \in \TT \mbox{ s.t. }  T \subsetneq T^\prime\}}_{=:\TB_2} 
\cup \underbrace{ (\TT \cap \TT^\prime) }_{=:\TB_3}. 
\end{align}
We refer to Lemma~\ref{lemma:fcc} for the proofs that the three sets in the definition of the finest common coarsening
are pairwise disjoint and that $\fcc(\TT,\TT^\prime)$ is indeed a regular triangulation of $\Omega$. \newline

Let $\widehat\TT_\ell$ be the uniform refinement of $\widehat \TT_0$ of level $\ell$. We call 
$\operatorname*{level}(T):=\ell$ the \emph{level}
of an element $T \in \widehat\TT_\ell$.
Given a regular triangulation $\TT$, we will consider
$$
\Tt_\ell:= \fcc( \TT , \widehat \TT_\ell),
$$
which is in general a coarser mesh than the uniform triangulation $\widehat\TT_\ell$.

\subsubsection{Adapted Scott-Zhang operators}
We recall the basic construction of the Scott-Zhang operator of 
\cite{scott-zhang90} or \cite[Sec.~{4.8}]{brenner-scott02}. It will be convenient
in the proof of Lemma~\ref{lemma:scott-zhang-tilde-hat} to use Lagrange bases of the 
space $S^{p,1}(\TT')$ defined on a mesh $\TT'$, where $\TT'$ is either $\Th_\ell$ or $\fcc(\TT,\Th_\ell)$. 
\begin{enumerate}
\item 
On the reference $d$-simplex $\widehat T = \operatorname*{conv}\{z_1,\dots,z_{d+1}\}$, let the 
$\operatorname{dim} P_p$ nodes ${\mathcal N}(\widehat T)$ 
be the regularly spaced nodes as described in \cite[Sec.~{2.2}]{ciarlet76a} (called ``principal lattice'' there),
\begin{align*}
 \mathcal{N}(\widehat T) := 
 \Big\{x=\sum_{j=1}^{d+1}\lambda_j z_j\,:\, \sum_{j=1}^{d+1} \lambda_j =1, \; 
 \lambda_j \in \Big\{\frac i p,i=0,\dots,p\Big\} \Big\}.
\end{align*}
We note that any polynomial in $P_p$ is uniquely determined by its values on $\mathcal{N}(\widehat T)$.

\item 
The nodes ${\mathcal N}(\TT')\subset \overline{\Omega}$ for the mesh $\TT'$ are the push-forward of the nodes 
of ${\mathcal N}(\widehat T)$ under the element maps. The Lagrange basis $\{\varphi_{z,\TT'}\,|\, z \in {\mathcal N}(\TT')\}$ 
of $S^{p,1}(\TT')$ (with respect to the nodes ${\mathcal N}(\TT')$) 
is characterized by $\varphi_{z,\TT'}(z') = \delta_{z,z'}$ for all $z$, $z' \in {\mathcal N}(\TT')$; 
here, $\delta_{z,z'}$ is the Kronecker Delta defined as $\delta_{z,z'}  = 1$ if $z = z'$ and $\delta_{z,z'} = 0$ for $z \ne z'$. 
\item 
The basis functions $\varphi_{z,\TT'}$ have the following support properties:  
a) if $z \in T$ for some $T \in \TT'$, then $\operatorname{supp} \varphi_{z,\TT'} \subset \overline{T}$; 
b) if $z \in f$ for some $j$-dimensional face ($j < d$) of $T$, then 
$\operatorname{supp} \varphi_{z,\TT'} \subset \omega_f$, where 
$\omega_f = \operatorname{int}\bigcup \{\overline{T}\colon f \mbox{ is $j$-face of $T \in \TT'$}\}$. 
In particular, if $z \not \in \overline{T}$, then $\operatorname{supp} \varphi_{z,\TT'}\cap T = \emptyset$. 
\item 
For each element $T\in\TT'$, one has a dual basis $\{\varphi^\ast_{z,T}\,:\, z \in \overline{T}\}\subset P_p(T)$ 
of $P_p(T)$, i.e., $\int_{T} \varphi^\ast_{z,T} \varphi_{z',\TT'} = \delta_{z,z'}$ for all nodes 
$z$, $z' \in \overline{T}$. 
In particular, this gives
\begin{equation}
\label{eq:averaging-gleich-punktauswertung} 
\int_{T} \varphi^\ast_{z,T} u \, dx = u(z) \qquad \forall T \in \TT' \quad \forall u \in P_p(T). 
\end{equation}
\item 
For each node $z \in {\mathcal N}(\TT')$, define the \emph{admissible set of averaging elements} 
as ${\mathcal A}(z,\TT'):=\{T \in \TT'\colon z \in \overline{T}\}$. A Scott-Zhang operator is 
then defined by selecting, for each $z$, a $T_z \in {\mathcal A}(z,\TT')$ and setting 
\begin{align}\label{eq:SZdef}
I^{SZ} u := \sum_{z\in{\mathcal N}(\TT')} \varphi_{z,\TT'} \left(\int_{T_z} \varphi^\ast_{z,T_z} u\, dx\right). 
\end{align}
\end{enumerate} 
For nodes $z$ that are on the boundary of an element, the admissible set ${\mathcal A}(z, \TT')$ has more than 
one element. However, from (\ref{eq:averaging-gleich-punktauswertung}), we get that the values 
of the functionals coincide on $S^{p,1}(\TT')$: 
\begin{equation}
\label{eq:averaging-egal}
\int_{T_z} \varphi^\ast_{z,T_z} u \; dx = u(z) = \int_{T'_z} \varphi^\ast_{z,T'_z} \,u \, dx 
\qquad \forall T_z, T'_z \in {\mathcal A}(z,\TT) \quad 
\forall u \in S^{p,1}(\TT'). 
\end{equation}
We also highlight that (\ref{eq:averaging-gleich-punktauswertung}) implies that
$I^{SZ}$ is a projection onto $S^{p,1}(\TT)$. 
Such Scott-Zhang operators satisfy the stability and approximation properties 
of Assumption~\ref{assumption:I-discontinuous} with constants that 
solely depend on $p$, the specific choice polynomial basis, 
the shape-regularity of the underlying triangulation, 
and $\Omega$. In particular, the constants are independent of the specific
choice of averaging region $T_z$. 

The freedom in the choice of the averaging element $T_z$ can be exploited to ensure
additional properties, see also \cite[Sec.~3]{diening-kreuzer-stevenson16},\cite[Sec.~4.3]{feischl2017optimal}.
A guiding principle in the following definition of our modified Scott-Zhang operator is 
that in the definition of $\widetilde I^{SZ}$ one selects the averaging element $T_z$ 
from the mesh $\TT$ whenever possible: 
\begin{definition}[adapted Scott-Zhang operators]
\label{def:adapted-SZ}
For given $\TT$ that is obtained by NVB-refinement from a regular triangulation
$\Th_0$ and $\Tt_\ell = \fcc(\TT, \Th_\ell)$, the operators 
$\widetilde I^{SZ}_\ell: L^2(\Omega) \rightarrow \widetilde V_\ell = S^{p,1}(\Tt_\ell)$ 
and 
$\widehat I^{SZ}_\ell: L^2(\Omega) \rightarrow \widehat V_\ell = S^{p,1}(\Th_\ell)$ 
are Scott-Zhang operators as defined in \eqref{eq:SZdef} with the following choice of 
averaging element $T_z$ for $\widetilde I^{SZ}_\ell$ and 
$\widehat I^{SZ}_\ell$: 
\begin{enumerate}[(1)]
\item
\label{item:def-SZ-1}
First, loop through all $T \in \Th_\ell \cap \Tt_\ell$ (in any fixed order) 
and select the averaging sets $T_z$ for the nodes $z \in \overline{T}$ as follows: 
  \begin{enumerate}[(a)]
    \item 
\label{item:def-SZ-1a}
    If $z \in T$, then select $T_z = T$ for both $\widehat I^{SZ}_\ell$ and $\widetilde I^{SZ}_\ell$. 
    \item 
\label{item:def-SZ-1b}
    If $z \in \partial T$ and the node $z$ has not been assigned an averaging set $T_z$ yet, then:  
     \begin{enumerate}[(i)]
        \item 
\label{item:def-SZ-1bi}
           If ${\mathcal A}(z,\Th_\ell)$ contains an element $T' \in \Th_\ell$ 
            that is a proper subset of an element $\widetilde T \in \TT$, then select this $T'$ to define
              $\widehat I^{SZ}_\ell$ and select $\widetilde T$ for the definition of $\widetilde I^{SZ}_\ell$. 
   \item 
   \label{item:def-SZ-1bii}  
     Else select $T$ for both $\widehat I^{SZ}_\ell$ 
     and $\widetilde I^{SZ}_\ell$.             
     \end{enumerate}
  \end{enumerate}
\item
\label{item:def-SZ-2}
Next, loop through all $T \in \Tt_\ell \setminus \Th_\ell$ (in any fixed order). 
Select, for the construction of $\widetilde I^{SZ}_\ell$, this $T$ as the 
averaging element for all nodes $z$ with $z \in \overline{T}$ that have 
not already been fixed in step (\ref{item:def-SZ-1}) or in a previous step of the loop. 
This completes the definition of $\widetilde I^{SZ}_\ell$. 
\item 
\label{item:def-SZ-3}
Finally, loop through all $T \in \Th_\ell \setminus \Tt_\ell$ (in any fixed order). 
Select, for the construction of $\widehat I^{SZ}_\ell$, this $T$ as the 
averaging element for all nodes $z$ with $z \in \overline{T}$ that have 
not already been fixed in step (\ref{item:def-SZ-1}) or in a previous step of the loop. 
This completes the definition of $\widehat I^{SZ}_\ell$. 
\end{enumerate}
\end{definition}
We note, that this definition of the adapted Scott-Zhang operators is exploited to show 
$\widehat I^{SZ}_\ell u = \widetilde I^{SZ}_\ell u $ for all $u \in S^{p,1}(\TT)$, which is 
proven in Lemma~\ref{lemma:scott-zhang-tilde-hat} below.

\subsubsection{The multilevel decomposition}
With the use of the adapted Scott-Zhang operators $\widetilde I^{SZ}_\ell$ and a mesh hierarchy
based on NVB meshes and the finest common coarsening between theses meshes and uniform refined meshes, 
we obtain a multilevel decomposition with norm equivalence in the Besov space $B^{3\theta /2}_{2,q}(\Omega)$
as a consequence of the stability estimate of Theorem~\ref{thm:stability-discontinuous}.

\begin{theorem}
\label{thm:multilevel-fcc}
Let $\TT$ be a mesh obtained by NVB-refinement of a triangulation $\Th_0$ with mesh size $\widehat h_0$. 
Let $\Th_\ell$ be the sequence of uniformly refined meshes starting from $\Th_0$
with mesh size $\hh_\ell = \hh_0 2^{-\ell}$. 
Set $\Tt_\ell:= \fcc(\TT,\Th_\ell)$.
Let $\widetilde{I}^{SZ}_\ell:L^2(\Omega) \rightarrow  S^{p,1}(\Tt_\ell)$ 
be the adapted Scott-Zhang operator 
defined in Definition~\ref{def:adapted-SZ}.
Then, on the space $S^{p,1}(\TT)$ 
the following three norms are equivalent with equivalence constants depending
only on $\Th_0$, $p$, $\theta \in (0,1)$, and $q \in [1,\infty]$: 
\begin{align*}
	&\|u\|_{B^{3\theta /2}_{2,q}(\Omega)}, \\
	& \| \widetilde{I}^{SZ}_0 u\|_{L^2(\Omega)} + \|(2^{3\theta \ell/2 } \|u - \widetilde{I}^{SZ}_\ell u\|_{L^2(\Omega)})_{\ell \ge 0}\|_{\ell^q}, \\
	& \| \widetilde{I}^{SZ}_0 u\|_{L^2(\Omega)} + \|(2^{3\theta \ell/2 } \|Q_\ell u\|_{L^2(\Omega)})_{\ell \ge 0}\|_{\ell^q},  
\end{align*}
where $Q_\ell = \widetilde{I}^{SZ}_{\ell+1} - \widetilde{I}^{SZ}_\ell$. 
\end{theorem}

\subsection{A realization of an optimal multilevel preconditioner for the fractional Laplacian}
The final main result of this paper presents a multilevel diagonal preconditioner with uniformly bounded condition
number on locally refined triangulations for the fractional Laplacian.

With the integral fractional Laplacian defined as the principal value integral
\begin{align*}
(-\Delta) ^su(x) := C(d,s) \,\text{P.V.} \int_{\mathbb{R} ^d} \frac{u (x)- u (y)}{|x-y|^ {d+2s}}dy \qquad C (d,s) := 2 ^{2s} s \frac{\Gamma (s + d / 2)}{\pi ^{d /2} \Gamma (1-s)},
\end{align*}
where $\Gamma (\,\cdot\,)$  denotes the Gamma function,
we consider the equation
\begin{align}\label{eq:modelproblem}
\nonumber (- \Delta) ^s u &=f \qquad \text{in}  \quad \Omega, \\
u &=0 \qquad \text{in}  \quad \Omega ^ c
\end{align}
for a given right-hand side $f \in H^{-s}(\Omega)$. Here, $H^{-s}(\Omega)$ denotes the dual space of the
Hilbert space 
\begin{align*}
\widetilde{H}^s(\Omega) = \{ u \in H^s (\mathbb{R}^d) \,: \, u \equiv 0 \,\, \text{on} \,\, \Omega ^c\},\qquad 
 \norm{v}_{\widetilde{H}^{s}(\Omega)}^2 := \norm{v}_{H^s(\Omega)}^2 + \norm{v/\rho^s}_{L^2(\Omega)}^2,
\end{align*}
where $\rho(x) = \operatorname{dist}(x,\partial\Omega)$ is the distance of a point $x \in \Omega$ to the boundary $\partial \Omega$.

The weak formulation of~\eqref{eq:modelproblem} is given by finding $u \in \widetilde{H}^s(\Omega)$ such that
\begin{align}\label{eq:weakform}
 a(u,v) 
 := \frac{C(d,s)}{2} \int\int_{\R^d\times\R^d} 
 \frac{(u(x)-u(y))(v(x)-v(y))}{\abs{x-y}^{d+2s}} \, dx \, dy = \int_{\Omega} fv \,dx 
 \quad \forall v \in \widetilde{H}^s(\Omega).
\end{align}
Existence and uniqueness of $u \in \widetilde{H}^s(\Omega)$ follow from
the Lax--Milgram lemma. \newline

With a given regular triangulation $\TT_0$, 
we consider two hierarchical sequence of meshes $\mathcal{T}_\ell$, $\widetilde{\mathcal{T}}_\ell$, $\ell = 0,\dots,L$:

\begin{enumerate}
 \item $\mathcal{T}_\ell$ is generated by an adaptive algorithm (see, e.g., \cite{Doer96})
of the form \texttt{SOLVE} -- \texttt{ESTIMATE} -- \texttt{MARK} -- \texttt{REFINE}, where the step \texttt{REFINE}
is done by newest vertex bisection. In the following,
both for the case of piecewise linear and piecewise constant basis function, 
we always assume that the meshes $\TT_\ell$ are regular in the sense of Ciarlet.

\item From a given triangulation $\TT_L$ obtained by NVB refinement of $\TT_0$ - e.g. given from an adaptive algorithm -
      the finest common coarsening with the uniform refinements of $\TT_0$, denoted by $\widehat\TT_\ell$, 
      provides a hierarchy of meshes 
      $\widetilde{\mathcal{T}}_\ell=\fcc(\mathcal{T}_L,\widehat\TT_\ell)$.
\end{enumerate}

 \subsubsection{A local multilevel diagonal preconditioner for adaptively refined meshes} 

 We start with the case of the adaptively generated mesh hierarchy $(\TT_\ell)_\ell$.
On the mesh $\TT_\ell$ we discretize with piecewise constants (for $0<s<1/2$)
in the space $V_\ell^0 = S^{0,0}(\mathcal{T}_\ell)$ 
and piecewise linears (for $0<s<1$) in the space
$V_\ell^1 = S^{1,1}_0(\mathcal{T}_\ell)$. If the distinction between $V_\ell^0$ and $V_\ell^1$ is not 
essential, we write $V_\ell$ meaning
$V_\ell \in \{V_\ell^0,V_\ell^1\}$. 
The Galerkin discretization  \eqref{eq:modelproblem} in $V_\ell$ of  reads as: 
Find $u_\ell \in V_\ell$, such that
 \begin{align}\label{eq:GalerkinDiscretization}
a(u_\ell,v_\ell) = \skp{f,v_\ell}_{L^2(\Omega)} \qquad  \forall v_{\ell} \in V_\ell.
 \end{align}
 Moreover, on the uniform refined meshes $\widehat \TT_\ell$, in the same way,
 we define the discrete spaces 
 \begin{align*}
 \widehat V_\ell^0 = S^{0,0}(\mathcal{T}_\ell), \quad \widehat V_\ell^1 = S^{1,1}_0(\mathcal{T}_\ell), \qquad \text{and}
 \quad 
 \widehat V_\ell \in \{ \widehat V_\ell^0, \widehat V_\ell^1\}.
 \end{align*}
 
 We define sets of ``characteristic'' points $\mathcal{N}_\ell^i$, $i=0,1$ representing the degrees of freedom of $V_\ell$. 
 For the  piecewise constant case $V_\ell^0$, the set $\mathcal{N}_\ell^0$
 contains all barycenters of elements of the mesh $\TT_\ell$. 
  For the piecewise linear case $V_\ell^1$, we 
 denote the set of all interior vertices of the mesh $\mathcal{T}_\ell$ by  $\mathcal{N}_\ell^1$.
 If the distinction between $\mathcal{N}_\ell^0$ and $\mathcal{N}_\ell^1$ is not essential, we 
 will write $\mathcal{N}_\ell$ meaning $\mathcal{N}_\ell \in \{\mathcal{N}_\ell^0,\mathcal{N}_\ell^1\}$
 is either $\mathcal{N}_\ell^0$ if $V_\ell = V_\ell^0$ or 
 $\mathcal{N}_\ell^1$ if $V_\ell = V_\ell^1$ and call $z \in \mathcal{N}_\ell$  nodes.

 We choose a basis of $V_\ell = \operatorname*{span}\{\varphi^\ell_{z_j}, z_j \in \mathcal{N}_\ell, j=1,\dots,N_\ell\}$ -- 
 for the piecewise constants we take the characteristic functions $\varphi_{z_j}^\ell = \chi_{T_j}$ of the element 
satisfying $z_j \in T_j \in \mathcal{T}_\ell$,
and for the piecewise linears we take hat functions corresponding to the interior nodes defined by 
$\varphi_{z_j}^\ell(z_i) = \delta_{j,i}$ for all nodes $z_i \in \mathcal{N}_\ell$. 
With these bases, we can write  $u_{\ell} =\sum_{j=1}^{N_\ell}\mathbf{x}^\ell_j \varphi_{z_j}^{\ell}$, and 
\eqref{eq:GalerkinDiscretization} is equivalent to solving the linear system
 \begin{align}\label{linear system}
\mathbf{A}^{\ell}\mathbf{x}^\ell=\mathbf{b}^{\ell}
 \end{align}
with the stiffness matrix $\mathbf{A}^\ell$ and load vector $\mathbf{b}^\ell$
 \begin{align}
 \mathbf{A}^{\ell}_{kj}:= a(\varphi_{z_j}^{\ell},\varphi_{z_k}^{\ell}), \qquad  \mathbf{b}_k^{\ell}:= \skp{f, \varphi_{z_k}^{\ell}}_{L^2(\Omega)}.
 \end{align}

 Again, we mention that the condition number of the unpreconditioned Galerkin matrix grows like 
 $  \kappa(\mathbf{A}^\ell) \sim N_\ell^{2s/d} \left(\frac{h^\ell_{\rm max}}{h^\ell_{\rm min}}\right)^{d-2s},$
 which stresses the need for a preconditioner in order to use an iterative solver.

 For fixed $L \in \N_0$, we  introduce a \emph{local multilevel diagonal preconditioner} 
 $(\mathbf{B}^L)^{-1}$ of BPX-type  
 for the stiffness matrix $\mathbf{A}^{L}$ from \eqref{linear system} in the same way as in
 \cite{feischl2017optimal,ainsworth2003multilevel}. That is, 
following  \cite{feischl2017optimal}, we define 
the patch of a node $z \in \mathcal{N}_\ell$ as 
$$\omega_\ell(z):=\operatorname*{interior}\bigcup\{\overline{T}\,:\, T \in \TT_\ell, z \in \overline{T}\}.$$ 

The sets $\mathcal{M}_\ell^i$, $i=0,1$, defined in the following describe the changes in the mesh hierarchy between the level 
$\ell$ and $\ell-1$ and are crucial for the definition of the local diagonal scaling.  
For the case of piecewise linears, we define the sets $\mathcal{M}_\ell^1$ as
the sets of new vertices and their direct neighbors in the mesh $\TT_\ell$: 
We set $\mathcal{M}_0^1:= \mathcal{N}_0^1$ and 
 \begin{align}
  \mathcal{M}_\ell^1 := \mathcal{N}_\ell^1 \backslash \mathcal{N}_{\ell-1}^1 
 \cup \{ z \in \mathcal{N}_\ell^1  \cap \mathcal{N}_{\ell-1}^1 \, :\, \omega_\ell(z) \subsetneqq \omega_{\ell-1}(z) \} \quad \ell \geq 1.
 \end{align}
For the case of a piecewise constant discretization, we define  
 the set $\mathcal{M}_\ell^0$ simply as the barycenters corresponding to the new elements, i.e.,
 $\mathcal{M}_\ell^0 := \mathcal{N}_\ell^0 \backslash \mathcal{N}_{\ell-1}^0$ for $\ell \geq 1$.
 In the same way as for the nodes $\mathcal{N}_\ell$ we write $\mathcal{M}_\ell$ to either 
 be $\mathcal{M}_\ell^0$ and $\mathcal{M}_\ell^1$, which should be clear from context.
 
The local multilevel diagonal preconditioner is given by
  \begin{align}\label{eq:LMDPrec}
  (\mathbf{B}^{L})^{-1}:= \sum_{\ell=0}^{L}\mathbf{I} ^{\ell}\; \mathbf{D}^{\ell}_{\rm inv} ( \mathbf{I} ^{\ell})^T,
  \end{align}
 where, with $N_\ell:= \# \mathcal{N}_\ell$, the appearing matrices are defined as
 \begin{itemize}
 	\item 
 $\mathbf{I}^{\ell} \in \mathbb{R}^{N_L \times  N_\ell}$ denotes the identity  matrix corresponds to  the embedding 
 $\mathcal{I}^{\ell} :V_\ell  \rightarrow V_L$.
 \item
 $\mathbf{D}^{\ell}_{\rm inv} \in \mathbb{R}^{N_\ell \times N_\ell}$ is a diagonal matrix with entries 
 $(\mathbf{D}^{\ell}_{\rm inv})_{jk} = \begin{cases}
 (\mathbf{A}^\ell_{jj})^{-1} \delta_{jk}  & j \; : \; z_j \in \mathcal{M}_\ell, \\
 0&\text{ otherwise}   
 \end{cases},$
i.e., the entries of the diagonal matrix are the inverse diagonal entries of the matrix $\mathbf{A}^\ell$ corresponding to the degrees of freedom in $\mathcal{M}_\ell$.
   \end{itemize}
    Moreover, we define the additive Schwarz matrix  
    $\mathbf{P}_{AS}^{L}: =   {(\mathbf{B}^{L})} ^{-1}\mathbf{A}^{L}$. 
Instead of solving \eqref{linear system} for $\ell=L$, we solve the following preconditioned linear systems
\begin{align}
\mathbf{P}_{AS}^{L}\mathbf{ x} ^{L} = {(\mathbf{B}^{L})} ^{-1} \mathbf{b} ^{L}. 
\end{align}
The following theorem is the main result of this section and provides the optimal bounds to the 
eigenvalues of the preconditioned matrix.
 \begin{theorem}\label{thm:ASfractional}
The minimal and maximal eigenvalues of the additive Schwarz  matrix $\mathbf{P}_{AS}^{L}$ are bounded by
 \begin{align}
 c \le \lambda _ {\min}\left( \mathbf{P}_{AS}^{L} \right)  \qquad \text{and} \qquad 
 \lambda _ {\max}\left(\mathbf{P}_{AS}^{L} \right) \le C,
 \end{align}
 where the constants $c$, $C>0$ depend only on $\Omega$, $d$, $s$, and the initial triangulation $\TT_0$.
 \end{theorem}
 
 \begin{remark}
  The preconditioner $ (\mathbf{B}^L)^{-1}$ is a symmetric positive definite matrix
  and the preconditioned matrix $\mathbf{P}_{AS}^L$ is symmetric and positive definite 
  with respect to the inner-product induced by $ \mathbf{B}^L$. 
  Therefore, Theorem~\ref{thm:ASfractional} leads to $\kappa(\mathbf{P}_{AS}^L) \leq C/c$.
\eremk
 \end{remark}
\begin{remark}
The cost to apply the preconditioner is proportional to $\sum_{\ell=0}^L \operatorname{card} \mathcal{M}_\ell
 = O(N_L)$ by \cite[Sec.~{3.1}]{feischl2017optimal}.
\eremk
\end{remark}
 \subsubsection{A local multilevel diagonal preconditioner using a finest common coarsening mesh hierarchy} 
\label{sec:LMD-fcc}
 In this subsection, we provide a result similar to Theorem~\ref{thm:ASfractional} for the meshes 
 $\widetilde{\mathcal{T}}_\ell=\fcc(\mathcal{T}_L,\widehat\TT_\ell)$, where $\ell=0,\dots, L$.

 With $\widetilde V_\ell^0 = S^{0,0}(\widetilde{\TT}_\ell)$, $\widetilde V_\ell^1 = S^{1,1}_0(\widetilde{\TT}_\ell)$, 
and $\widetilde V_\ell \in \{\widetilde V_\ell^0,\widetilde V_\ell^1\}$ 
 being either the piecewise 
 constants or piecewise linears on $\widetilde{\mathcal{T}}_\ell$, the Galerkin discretization 
 of finding $\widetilde u_\ell \in \widetilde V_\ell$ such that 
  \begin{align}\label{eq:GalerkinDiscretization2}
a(\widetilde u_\ell,\widetilde v_\ell) = \skp{f,\widetilde v_\ell}_{L^2(\Omega)} \qquad 
\forall\; \widetilde v_{\ell} \in  \widetilde V_\ell
 \end{align}
 is equivalent to solving the linear system 
  \begin{align}\label{eq:linearSystem2}
\mathbf{\widetilde A}^{\ell}\widetilde{\mathbf{x}}^\ell=\widetilde{\mathbf{b}}^{\ell} \qquad 
 \end{align}
by choosing a nodal basis as in the previous subsection.
The set of nodes $\widetilde{\mathcal{N}}_\ell^i$, $i=0,1$ and $\widetilde{\mathcal{N}}_\ell$ as well as the 
sets $\widetilde{\mathcal{M}}_\ell^{i}$, $i=0,1$ and $\widetilde{\mathcal{M}}_\ell$ can be defined in 
exactly the same way as in the previous subsection by just replacing the meshes $\TT_\ell$ with $\widetilde\TT_\ell$.
Therefore, we can define the local multilevel diagonal 
preconditioner 
  \begin{align*}
  (\widetilde{\mathbf{B}}^{L})^{-1}:= \sum_{\ell=0}^{L}\mathbf{I} ^{\ell}\; \widetilde{\mathbf{D}}^{\ell}_{\rm inv} ( \mathbf{I} ^{\ell})^T
  \end{align*}
  in exactly the same way as in \eqref{eq:LMDPrec}.

  The following theorem then gives optimal bounds for the smallest and largest eigenvalues of the 
  preconditioned matrix $\widetilde{\mathbf{P}}_{AS}^{L}:= {(\widetilde{\mathbf{B} }^{L})} ^{-1}
  \widetilde{\mathbf{A}}^{L}$.
  
   \begin{theorem}\label{thm:ASfractional2}
The minimal and maximal eigenvalues of the additive Schwarz  matrix $\widetilde{\mathbf{P}}_{AS}^{L}$ are bounded by
 \begin{align}
 c \le \lambda _ {\min}\left( \widetilde {\mathbf{P}}_{AS}^{L} \right)  \qquad \text{and} \qquad \lambda _ {\max}\left( \widetilde {\mathbf{P}}_{AS}^{L} \right) \le C,
 \end{align}
 where the constants $c$, $C>0$ depend only on $\Omega$, $d$, $s$, and the initial triangulation $\TT_0$.
 \end{theorem}
 \begin{remark} 
By Lemma~\ref{lemma:chopping} the cost of the preconditioner are, up to a constant,   
$\operatorname{card} \widetilde{\mathcal M}_0 + \sum_{\ell=1}^L \operatorname{card}\widetilde{\mathcal M}_\ell\lesssim 
\operatorname{card} \widetilde{\mathcal M}_0 + \sum_{\ell=0}^L \operatorname{card}\widetilde{\mathcal N}_\ell - \operatorname{card} \widetilde{\mathcal N}_{\ell-1}
\lesssim \operatorname{card} \widetilde{\mathcal N}_L = \operatorname{card} \TT$. 
\eremk
 \end{remark} 
\section{Stability of Scott-Zhang type operators}
\label{sec:SZproof}

We will need mollifiers with certain local approximation properties. Essentially, such operators are given by 
those classical mollifiers that reproduce, or at least approximate to high order, polynomials of degree $p$. 
The following proposition, which is taken from \cite{karkulik-melenk15}, provides such operators. Our primary reason
for working with this particular class of approximation operators is that the technical complications associated
with the boundary of $\partial \Omega$ have been taken care of. 

\begin{proposition}[\protect{\cite[Thm.~{2.3}]{karkulik-melenk15}}]
\label{prop:KM15}
Let $\Omega$ be a bounded Lipschitz domain and $p \in \BbbN_0$ be fixed. Let $\omega \subset \Omega$ be an 
arbitrary open set and denote for $\varepsilon > 0$ by 
$\omega_\varepsilon:= \Omega \cap \cup_{x \in \omega} B_\varepsilon(x)$ the
 ``$\varepsilon$-neighborhood'' of $\omega$. 
Then, there exists a constant $C > 0$ such that 
for every $\varepsilon > 0$ there is a linear operator $\average{\varepsilon}{}:L^1_{loc}(\Omega) \rightarrow 
C^\infty(\overline{\Omega})$ with the stability and approximation properties
\begin{enumerate}[(i)]
\item 
\label{item:prop:KM15-i}
If $u \in H^k(\omega_\varepsilon)$ with $k \leq p+1$, then 
$\|\average{\varepsilon}{u}\|_{H^\ell(\omega)} \leq C \varepsilon^{-\ell + k} \|u\|_{H^k(\omega_\varepsilon)}, 
\qquad \ell = k,\ldots,p+1$. 
\item 
\label{item:prop:KM15-ii}
If $u \in H^k(\omega_\varepsilon)$ with $k \leq p+1$, then 
$\|u - \average{\varepsilon}{u}\|_{H^\ell(\omega)} \leq C \varepsilon^{k-\ell} \|u\|_{H^k(\omega_\varepsilon)}, 
\qquad \ell = 0,\ldots,k$. 
\end{enumerate}
\end{proposition}
\begin{proof}
The proof for the much more technical case of a {\em variable} length scale function $\varepsilon = \varepsilon(x)$ 
is given in \cite[Thm.~{2.3}]{karkulik-melenk15}. We give the idea of the proof: in the interior of 
$\Omega$, the operator $\average{\varepsilon}{}$ has the form $\average{\varepsilon}{u} = u \ast \rho_\varepsilon$,
where the mollifier $\rho_\varepsilon$ is such that it reproduces polynomials of degree $p$ (the ``classical''
mollifier reproduces merely constant functions). Near the boundary, this standard averaging is modified such 
that $\average{\varepsilon}{u}(x)$ is not obtained by averaging $u$ on $B_{\varepsilon}(x)$ but by 
averaging $u$ on the ball $B_{\varepsilon}(x + \varepsilon b)$ and evaluating the Taylor polynomial
of degree $p$ of this averaged function at the point $x$ of interest; the vector $b$ is suitable of size $O(1)$
and it ensures that the averaging is performed inside $\Omega$. 
\end{proof}

With the mollifiers from Proposition~\ref{prop:KM15}, we can prove stability and approximation
properties for operators satisfying Assumption~\ref{assumption:I-discontinuous} in stronger norms.
\begin{lemma}
\label{lemma:approximation}
Let $m \in \{1,2\}$ and $p \ge m-1$. Assume that the linear operator 
$\II:H^m(\Omega) \rightarrow S^{p,m-1}(\TT)$ satisfies 
Assumption~\ref{assumption:I-discontinuous}. Then, for all $T \in \TT$ the stability
\begin{equation}
\label{eq:discontinuous-H1-stability}
\|\II u\|_{H^j(T)} \leq C \|u\|_{H^j(\omega^2 (T))},
\qquad j=0,\ldots,m
\end{equation} 
and approximation property
\begin{equation}
\label{eq:lemma:approximation-10}
\|u - \II u\|_{H^r(T)} \leq C h_T^{k-r}\|u\|_{H^k(\omega^2(T))}, 
\qquad r=0,\ldots,\min\{k,m\}
\end{equation}
hold, where the constants $C>0$ depend solely on $d$, $m$, $p$, and the $\gamma$-shape-regularity of $\TT$. 
\end{lemma}
\begin{proof}
Let $T \in \TT$ be arbitrary.
We use the operator ${\mathcal A}_\varepsilon$ of Proposition~\ref{prop:KM15}
with $\omega=\omega(T)$ and $\varepsilon \sim h_T$, such that $\omega_\varepsilon \subset \omega^2(T)$. 
We write using the triangle inequality
\begin{align*}
\|u - \II u\|_{H^r(T)} &\leq 
\|u - {\mathcal A}_\varepsilon u\|_{H^r(T)} +  
\|{\mathcal A}_\varepsilon u - \II {\mathcal A}_\varepsilon u\|_{H^r(T)} +  
\|\II( u - {\mathcal A}_\varepsilon u)\|_{H^r(T)}  \\
&=:T_1 + T_2 + T_3. 
\end{align*}
By Proposition~\ref{prop:KM15} we have 
$T_1 \lesssim h_T^{k-r} \|u\|_{H^k(\omega^2(T))}$. 
A polynomial inverse estimate, see, e.g., \cite{dahmen2004inverse}, the stability property 
\eqref{item:assumption:I-discontinuous-ii} of Assumption~\ref{assumption:I-discontinuous},
and Proposition~\ref{prop:KM15} give 
$$
T_3 \lesssim h_T^{-r} \|u - {\mathcal A}_\varepsilon u\|_{L^2(\omega(T))} 
\lesssim h_T^{-r} h_T^k \|u\|_{H^k(\omega^2(T))}. 
$$
In order to estimate $T_2$, we use a piecewise polynomial $q \in S^{p,m-1}(\TT) $ with approximation 
properties in the $H^r$-norm (e.g., a Cl{\'e}ment or Scott-Zhang type interpolation)
as given by \cite[Thm.~{4.8.12}]{brenner-scott02}.
Then, 
\begin{align*}
T_2 &\leq \|{\mathcal A}_\varepsilon u - u\|_{H^r(T)} + 
          \|u - q \|_{H^r(T)} + 
          \|\II {\mathcal A}_\varepsilon u - q\|_{H^r(T)}  =: T_{2,1} + T_{2,2} + T_{2,3}. 
\end{align*}
We already have estimated $T_{2,1} = T_1$. 
By \cite[Thm.~{4.8.12}]{brenner-scott02}
(and inspection of the procedure there), we obtain $T_{2,2} \lesssim h_T^{s-r} \|u\|_{H^s(\omega^2(T))}$. 
Finally, for $T_{2,3}$ we use an inverse estimate
$$
T_{2,3} \lesssim h_T^{-r} \|\II {\mathcal A}_\varepsilon u - q\|_{L^2(T)}
\lesssim 
h_T^{-r} \left[ \|\II {\mathcal A}_\varepsilon u - {\mathcal A}_\varepsilon u\|_{L^2(T)} + \|{\mathcal A}_\varepsilon u - u\|_{L^2(T)} + 
\|u  - q\|_{L^2(T)} \right].
$$
The last two terms have the desired form due to Proposition~\ref{prop:KM15} and \cite[Thm.~{4.8.12}]{brenner-scott02}.
For the remaining term, we write with Assumption~\ref{assumption:I-discontinuous} \eqref{item:assumption:I-discontinuous-iii}
and Proposition~\ref{prop:KM15}
\begin{align*}
\|\II {\mathcal A}_\varepsilon u - {\mathcal A}_\varepsilon u\|_{L^2(T)} 
\lesssim h_T^m \|{\mathcal A}_\varepsilon u\|_{H^m(\omega(T))}
\lesssim h_T^{m} h_T^{s-m}\|u\|_{H^s(\omega^2(T))}.
\end{align*}
Finally, 
(\ref{eq:discontinuous-H1-stability}) follows from 
(\ref{eq:lemma:approximation-10}) by selecting $r = s$. 	
\end{proof}

The generalization of Proposition~\ref{prop:KM15} to the case of variable length scale functions 
from \cite[Thm.~{2.3}]{karkulik-melenk15} can also be used to derive a smooth operator with approximation 
and stability properties for $h$-weighted and fractional norms.

\begin{corollary}
\label{cor:H12-estimate-scott-zhang}
With the mesh size function $h$ of $\TT$ and $t>0$, we define the function $\overline{h}:= \max\{t,h\}$.
Let $m$, $n \in \BbbN_0$ be fixed and $u \in H^m(\Omega)$.
Then, for every $t > 0$ there exists 
a linear operator $J_t:L^2(\Omega) \rightarrow C^\infty(\overline\Omega)$ with the stability  
\begin{align}
\label{eq:cor:H12-estimate-scott-zhang-10}
\|\overline{h}^{n} \nabla^{m+n} J_t u\|_{L^2(\Omega)} &\leq C_{m,n} \|u\|_{H^m(\Omega)} 
\end{align}
and approximation
properties 
\begin{align}
\label{eq:cor:H12-estimate-scott-zhang-15}
\sum_{j=0}^m \|\overline{h}^{-(j-m)} \nabla^j (u - J_t u)\|_{L^2(\Omega)} &\leq C_{m} \|u\|_{H^m(\Omega)}. 
\end{align}
In particular, interpolation arguments give
\begin{align}
\label{eq:cor:H12-estimate-scott-zhang-20}
\|\overline{h}^{1/2} \nabla J_t u\|_{L^2(\Omega)} + \|\overline{h}^{-1/2}( u - J_tu)\|_{L^{2}(\Omega)} & \leq C \|u\|_{H^{1/2}(\Omega)},  \\
\label{eq:cor:H12-estimate-scott-zhang-25}
\|\overline{h}^{1/2} \nabla^2 J_t u\|_{L^2(\Omega)} +  
\|\overline{h}^{-3/2}( u - J_tu)\|_{L^{2}(\Omega)}  + 
\|\overline{h}^{-1/2}\nabla ( u - J_tu)\|_{L^{2}(\Omega)} 
& \leq C \|u\|_{H^{3/2}(\Omega)}. 
\end{align}
The constants $C_{m,n}$ and $C_m$ depend on $m$ and $n$ as indicated, 
as well as on $\Omega$ and the $\gamma$-shape regularity of $\TT$. 
The constant $C$ depends only on $\Omega$ and the $\gamma$-shape regularity of $\TT$. 
\end{corollary}
\begin{proof}
\emph{1.~step:} For $t \ge \operatorname{diam} \Omega$, one may 
select $J_t  = 0$. 

\emph{2.~step:} For $t \leq \operatorname{diam}\Omega$
one constructs a length scale function $\varepsilon$ 
with $\varepsilon \sim \overline{h}$ in the following way: 
First, by mollification of  the piecewise constant function $h$ 
(see \cite[Lemma~{3.1}]{karkulik-melenk15} for details), one obtains
a function $\widetilde h \in C^\infty(\overline{\Omega})$ whose Lipschitz constant
${\mathcal L}$ depends solely on the $\gamma$-shape regularity of $\TT$ and $\Omega$. 
Next, one defines the auxiliary length scale function 
$\widetilde  \varepsilon(x):= \widetilde h(x) + t$. We note that 
the Lipschitz constant of $\widetilde \varepsilon$ is still ${\mathcal L}$. 
{}From \cite[Lemma~{5.7}]{karkulik-melenk15}, there are parameters 
$0 < \alpha < \beta$ (depending on ${\mathcal L}$) and $N_d \in \BbbN$
(depending only on the spatial dimension $d$) as well as 
closed balls $B_{ij}:= \overline{B}_{\alpha \widetilde \varepsilon(x_{ij})}(x_{ij})$, 
$i=1,\ldots,N_d$, $j \in \BbbN$ such that the following holds: 
\begin{enumerate}[(a)]
\item 
\label{item:overlap-a}
$\Omega \subset \cup_{i=1}^{N_d} \cup_{j\in\BbbN} B_{ij};$
\item 
\label{item:overlap-b}
There is a constant $C_{\rm big} > 0$, such that for each 
$i \in \{1,\ldots,N_d\}$ the stretched balls \linebreak
$\widehat B_{ij}:= \overline{B}_{\beta \widetilde \varepsilon(x_{ij})}(x_{ij})$ satisfy an overlap condition: 
$\# \{ j'\,|\, \widehat {B}_{ij'} \cap \widehat {B}_{ij} \ne \emptyset\} \leq C_{\rm big}$ for all $j \in \BbbN$. 
\item 
\label{item:overlap-c}
For pairs $(i,j)$ and $(i',j')$ with 
$\widehat B_{ij} \cap \widehat B_{i'j'} \ne \emptyset$  there holds 
$\widetilde \varepsilon(x_{ij}) \sim \widetilde \varepsilon(x_{i'j'})$ with 
implied constant depending solely on ${\mathcal L}$ and $\beta$. 
This implies {\sl a fortiori} that for 
pairs $(i,j)$ and $(i',j')$ with 
$B_{ij} \cap B_{i'j'} \ne \emptyset$  there holds 
$\widetilde \varepsilon(x_{ij}) \sim \widetilde \varepsilon(x_{i'j'})$ with 
implied constant depending solely on ${\mathcal L}$ and $\beta$
(which follows by inspection of the proof of 
\cite[Lemma~{5.7}]{karkulik-melenk15}).
\end{enumerate}
Denoting by $\chi_A$ the characteristic function of the set $A$, we define
the desired length scale function $\varepsilon$ as 
\begin{equation}
\label{eq:length-scale-fct}
\varepsilon:= 
\sum_{i=1}^{N_d} \sum_{j \in \BbbN} \widetilde \varepsilon(x_{ij}) (\chi_{B_{ij}} 
\ast \rho_{(\beta-\alpha)\widetilde \varepsilon(x_{ij})}), 
\end{equation}
where $\rho_\delta$ is a standard non-negative mollifier supported by 
$B_\delta(0)$. Let $x \in \Omega$. Due to \eqref{item:overlap-a} there is $(i,j)$ with 
$x \in B_{ij}$. The non-negativity of the mollifier $\rho_\delta$
gives $\varepsilon(x) \gtrsim \widetilde \varepsilon(x_{ij})$. 
Furthermore, (\ref{item:overlap-b}), (\ref{item:overlap-c}) we imply 
that the sum (\ref{eq:length-scale-fct}) is locally finite (with 
at most $N_d C_{\rm big}$ non-zero terms). In view of (\ref{item:overlap-c})
we get $\varepsilon(x) \lesssim \varepsilon(x_{ij})$. By studying 
derivatives of $\varepsilon$, we recognize that it is a length scale
function in the sense of \cite[Def.~{2.1}]{karkulik-melenk15}. 

\emph{3.~step:} The upshot of \cite[Lemma~{5.7}]{karkulik-melenk15} 
is that, once a length scale function $\varepsilon$ is available, then 
a covering argument can be employed. That is, the operator 
${\mathcal A}_\varepsilon$ of \cite[Thm.~{2.3}]{karkulik-melenk15} 
yields 
\begin{align*}
\sum_{j=0}^{m} \|\varepsilon^{m-j} \nabla^j (u - {\mathcal A}_\varepsilon u)\|_{L^2(\Omega)} & \lesssim \|u\|_{H^m(\Omega)}, \\
\|\varepsilon^n \nabla^{m+n}  {\mathcal A}_\varepsilon u\|_{L^2(\Omega)} &\lesssim 
\|u\|_{H^{m}(\Omega)}, 
\end{align*}
which proves (\ref{eq:cor:H12-estimate-scott-zhang-10}) and (\ref{eq:cor:H12-estimate-scott-zhang-15})
since $\varepsilon \sim \overline{h}$.

\emph{4.~step:}
Interpolation between the inequalities for $m=0$ and $m=1$ using \cite[Lemma~{23.1}]{tartar07} then gives the 
estimate (\ref{eq:cor:H12-estimate-scott-zhang-20}), 
and interpolation between $m=1$ and $m=2$ (\ref{eq:cor:H12-estimate-scott-zhang-25}).
\end{proof}

\begin{remark}
If the shape-regular mesh $\TT$ is obtained by repeated NVB from a coarse grid
$\TT_0$, then a simpler proof is possible: one may then construct 
a quasi-uniform mesh of mesh size $\sim t$ and consider 
$\widetilde \TT:= \fcc(\TT,\TT_t)$. Then, $J_t$ can be taken as a mollifier
of the standard Scott-Zhang operator associated with $\widetilde \TT$. 
\eremk
\end{remark}
\subsection{Proof of Theorem~\ref{thm:stability-discontinuous}}
\label{sec:proofSZ}

\begin{proof}[Proof of Theorem~\ref{thm:stability-discontinuous}]
The function $\II u$ is piecewise smooth on a finite mesh. Hence, it is an element of 
$B^{m-1/2}_{2,\infty}(\Omega)$,  so that only the 
stability estimate has to be proved. This is achieved by constructing an element 
$u_t:={\mathcal A}_{\delta t}(\II u)$ for an appropriate $\delta>0$ such that the 
$K$-functional can be estimated by the $H^{m-1/2}$-norm of $u$. We have
\begin{align}\label{eq:KfunctionalEstimate}
\norm{\II u}_{B_{2,\infty}^{m-1/2}(\Omega)} &= \sup_{t>0} t^{-1/2} K(t,\II u) \nonumber \\ &\lesssim 
\sup_{t>0} t^{-1/2} \left( \norm{\II u - {\mathcal A}_{\delta t}(\II u)}_{H^{m-1}(\Omega)} + 
t \norm{{\mathcal A}_{\delta t}(\II u)}_{H^m(\Omega)}\right).
\end{align}
With the operator $J_t$ from Corollary~\ref{cor:H12-estimate-scott-zhang}, we further decompose 
$u = (u - J_tu) + J_tu =: u_0 + u_1$ into an element of $H^{m-1}(\Omega)$ and one in $H^m(\Omega)$. 
By the triangle inequality, we have 
to control the right-hand side of \eqref{eq:KfunctionalEstimate} for both contributions separately.

{\em 1.~step:} For fixed $t>0$, we split the mesh into elements of size smaller and larger than $t$, 
\begin{align*}
{\mathcal T}_{\leq t} := \{T \in {\mathcal T} \colon \diam T \leq t\}, 
\qquad 
{\mathcal T}_{> t} := \{T \in {\mathcal T} \colon \diam T > t\}, 
\end{align*}
and define the regions covered by these elements by 
\begin{align}
\label{eq:thm:stability-discontinuous-20}
\Omega_{\leq t} := \operatorname*{interior} \left(\bigcup_{T \in {\mathcal T}_{\leq t}} \overline T\right), 
\qquad 
\Omega_{> t}:= \operatorname*{interior} \left(\bigcup_{T \in {\mathcal T}_{> t}} \overline T\right). 
\end{align} 
There is a constant $\delta > 0$, depending solely on the $\gamma$-shape regularity 
of ${\mathcal T}$, such that the ``$\delta t$-neighborhood'' 
$T_{\delta t}:= \Omega \cap \cup_{x \in T} B_{\delta t}(x)$ of each element in ${\mathcal T}_{> t}$ 
is contained in the patch of the element, i.e., 
$$
T_{\delta t} \subset \omega (T) \qquad \forall T \in {\mathcal T}_{> t}. 
$$
Moreover, for each $T \in {\mathcal T}_{> t}$, we define the \emph{inside strip} $S_{T,\delta t}$ at the boundary $\partial T$ of $T$ by 
\begin{align}
\label{eq:thm:stability-discontinuous-160}
S_{T,\delta t} &:= \{x \in T \colon \dist{x}{\partial T} < \delta t\}. 
\end{align}
%
Concerning the set ${\mathcal T}_{\leq t}$, we claim the existence of $\eta \ge \delta$, $C > 0$ depending only
on the $\gamma$-shape regularity of ${\mathcal T}$ such that the extended set 
\begin{equation*}
\Omega_{\eta t}:= \Omega \cap \bigcup_{x \in \Omega_{\leq t}} B_{\eta t}(x) 
\end{equation*}
satisfies the conditions
\begin{align}
\label{eq:thm:stability-discontinuous-40}
T \in {\mathcal T}_{\leq t}  &\Longrightarrow \omega^2 (T) \subset \Omega_{\eta t}, \\
\label{eq:thm:stability-discontinuous-50}
T \in {\mathcal T}  \mbox{ with }T \subset \Omega_{\eta t} &\Longrightarrow \diam T \leq C t,\\
	\label{eq:thm:stability-discontinuous-50a}
T \in {\mathcal T}  \mbox{ with } T \cap \Omega_{\eta t}\neq \emptyset &\Longrightarrow \omega(T) 
	\subset \Omega_{c\eta t}.
\end{align}
	where $c >0$ is a constant depending solely on the $\gamma$-shape regularity 
	of ${\mathcal T}$.
The choice of $\eta$ is dictated by the requirement 
(\ref{eq:thm:stability-discontinuous-40}). We note that the $\gamma$-shape regularity of ${\mathcal T}$ ensures that 
for all $T \in {\mathcal T}_{\leq t}$ the diameters of all elements $T^\prime \subset \omega (T)$ are bounded 
by $\widehat C t$ for some $\widehat C > 0$ depending only on $\gamma$. 
This implies (\ref{eq:thm:stability-discontinuous-40}) if $\eta$ is chosen sufficiently large. 

To see  (\ref{eq:thm:stability-discontinuous-50}), 
it suffices to consider elements $T \in {\mathcal T}$ with 
$T \subset \Omega_{\eta t} \setminus \Omega_{\leq t}$. 
Let $m_T$ be the center of the largest inscribed sphere in $T$ 
and note that the radius $\rho_T$ of that sphere 
is comparable to the element diameter $h_T$. Let $\widetilde m_T \in \overline{\Omega_{\leq t}}$ satisfy
$\dist{m_T}{\Omega_{\leq t}}  = \dist{m_T}{\widetilde m_T}$.
By definition of $\Omega_{\eta t}$, we have 
$m_T \in B_{\eta t}(\widetilde m_T)$ and by $T \subset \Omega_{\eta t} \setminus \Omega_{\leq t}$ that 
$B_{\rho_T}(m_T) \subset \Omega_{\eta t} \setminus \Omega_{\leq t}$. Thus, 
$$h_T \sim \rho_T \leq \dist{m_T}{\Omega_{\leq t}}  = \dist{m_T}{\widetilde m_T} \leq \eta t,$$
which proves (\ref{eq:thm:stability-discontinuous-50}). 

With the sets from \eqref{eq:thm:stability-discontinuous-20} and \eqref{eq:thm:stability-discontinuous-160}, 
we decompose for $k \in \N_0$ and $v \in H^k(\Omega)$ 
\begin{align}\label{eq:domainDecomp}
 \norm{v}_{H^{k}(\Omega)}^2 &\lesssim   \norm{v}_{H^{k}(\Omega_{\leq t})}^2 +  \norm{v}_{H^{k}(\Omega_{> t})}^2 \nonumber \\
  &\lesssim \norm{v}_{H^{k}(\Omega_{\leq t})}^2 + 
  \sum_{T \in \TT_{>t}} \norm{v}_{H^{k}(T\backslash S_{T,\delta t})}^2 
  + \sum_{T \in \TT_{>t}} \norm{v}_{H^{k}(S_{T,\delta t})}^2.
\end{align}
We employ this decomposition in \eqref{eq:KfunctionalEstimate} for 
$k = m-1$ and $v=\II u_i - {\mathcal A}_{\delta t}(\II u_i)$ as well as for $k=m$ and 
$v={\mathcal A}_{\delta t}(\II u_i)$ and $i \in \{0,1\}$.

In the following, we estimate all these contributions separately by the desired $H^{m-1/2}(\Omega)$-norm of $u$. 
The main ideas are that, on $\Omega_{\leq t}$, we exploit that elements are small. On 
$T\backslash S_{T,\delta t}$, we may exploit that a sufficiently small neighborhood of this set is still contained in 
$T$, and we can use the smoothness of $\II u_i$ inside $T$. For $S_{T,\delta t}$, we exploit the thinness of the strip.
\newline

{\em 2.~step:} We estimate $\II u_i - {\mathcal A}_{\delta t}(\II u_i)$ on $\Omega_{\leq t}$, where
$\delta \leq \eta $ is given by step~1. 

For $i=0$, we use the stability estimates of Proposition~\ref{prop:KM15}
and Lemma~\ref{lemma:approximation} and finally Corollary~\ref{cor:H12-estimate-scott-zhang} 
(using $\overline{h}\sim t$ due to \eqref{eq:thm:stability-discontinuous-50})
to obtain 
\begin{align*}
\|\II u_0 - \average{\delta t}{(\II u_0)}\|_{H^{m-1}(\Omega_{\leq t})} &\leq 
\|\II u_0\|_{H^{m-1}(\Omega_{\leq t})} + \|\average{\delta t}{(\II u_0)}\|_{H^{m-1}(\Omega_{\leq t})} \\
&\lesssim
\|\II u_0\|_{H^{m-1}(\Omega_{\leq t})} 
+ \|\II u_0\|_{H^{m-1}(\Omega_{\eta t})}  \\
&\stackrel{(\ref{eq:discontinuous-H1-stability})}{\lesssim}
\|u_0\|_{H^{m-1}(\Omega_{c\eta t})}
= \|u - J_t u\|_{H^{m-1}(\Omega_{c\eta t})} \\
&\stackrel{\rm Cor.~\ref{cor:H12-estimate-scott-zhang}}{\lesssim}t^{1/2}
\norm{u}_{H^{m-1/2}(\Omega)}. 
\end{align*}

For $i=1$, we use the approximation property of $\II$ 
(cf.\ (\ref{eq:lemma:approximation-10}) with $r = m-1$ and $k = m$) 
together with the fact that the element size 
of elements in $\Omega_{\leq t}$ is bounded by $t$ as well 
as the local stability and approximation properties of 
${\mathcal A}_{\delta t}$ from Proposition~\ref{prop:KM15} to get 
\begin{align*}
& \|\II u_1 - \average{\delta t}{(\II u_1)}\|_{H^{m-1}(\Omega_{\leq t})}  \\
&\qquad \leq \|\II u_1 - u_1\|_{H^{m-1}(\Omega_{\leq t})} + 
\|u_1 - \average{\delta t}{u_1}\|_{H^{m-1}(\Omega_{\leq t})} + 
\|\average{\delta t}{(u_1 - \II u_1)}\|_{H^{m-1}(\Omega_{\leq t})}  \\
&\qquad \stackrel{ h \lesssim t}{\lesssim }
t \|u_1\|_{H^m(\Omega_{\eta t})} + 
t \|u_1\|_{H^m(\Omega_{\eta t})} + 
\|u_1 - \II u_1\|_{H^{m-1}(\Omega_{\eta t})} \\
&\qquad\stackrel{h \lesssim t}{\lesssim} t \|u_1\|_{H^m(\Omega_{c\eta t})}
= t \|J_t u\|_{H^m(\Omega_{c\eta t})} \lesssim t^{1/2} \norm{u}_{H^{m-1/2}(\Omega)},  
\end{align*}
where the last step follows from Corollary~\ref{cor:H12-estimate-scott-zhang}. \newline

{\em 3.~step:} We estimate of ${\mathcal A}_{\delta t}(\II u_i)$ on $\Omega_{\leq t}$. 
For $i=0$, we estimate using the stability properties of the smoothing operator 
from Proposition~\ref{prop:KM15}, the stability of $\II$, and Corollary~\ref{cor:H12-estimate-scott-zhang} 
\begin{align*}
t\|\average{\delta  t}{(\II u_0)}\|_{H^m(\Omega_{\leq t})} 
&\lesssim \|\II u_0 \|_{H^{m-1}(\Omega_{\eta t})}  
\stackrel{(\ref{eq:discontinuous-H1-stability})}{\lesssim}
\|u_0\|_{H^{m-1}(\Omega_{c\eta t}))} \lesssim t^{1/2}\|u\|_{H^{m-1/2}(\Omega)}.
\end{align*}
Similarly, for $u_1 \in H^m(\Omega)$, we obtain
\begin{align*}
\nonumber 
t\|\average{\delta t}{(\II u_1)}\|_{H^m(\Omega_{\leq t})} &\lesssim 
t\|\average{\delta t}{(\II u_1 - u_1)}\|_{H^m(\Omega_{\leq t})}  + 
t\|\average{\delta t}{u_1}\|_{H^m(\Omega_{\leq t})}   \\
\nonumber 
&\stackrel{\text{Prop.~\ref{prop:KM15}}}{\lesssim}\|\II u_1 - u_1\|_{H^{m-1}(\Omega_{\eta t})} 
+ t\|u_1\|_{H^m(\Omega_{\eta t})}  \\
&
\stackrel{(\ref{eq:discontinuous-H1-stability}), h \le t}{\lesssim} t\|u_1\|_{H^m(\Omega_{c \eta t})} \lesssim t^{1/2}\|u\|_{H^{m-1/2}(\Omega)},  
\end{align*}
where the last step again follows from the definition of $u_1$ and Corollary~\ref{cor:H12-estimate-scott-zhang}. \newline

{\em 4.~step:} We derive estimates on $T\backslash S_{T,\delta t}$ for $T \in {\mathcal T}_{>t}$. Since the 
``$\delta t$-neighborhood'' $(T\backslash S_{T,\delta t})_{\delta t}$ of $T\backslash S_{T,\delta t}$ satisfies
$(T\backslash S_{T,\delta t})_{\delta t} \subseteq T$, Proposition~\ref{prop:KM15} and an
inverse inequality imply
\begin{align*}
  \|\II u_0 - \average{\delta t}{(\II u_0)}\|_{H^{m-1}(T \setminus S_{T,\delta t})} &\lesssim  t \|\II u_0\|_{H^m(T)} 
\lesssim t h_T^{-1} \|\II u_0\|_{H^{m-1}(T)}\\
&\stackrel{(\ref{eq:discontinuous-H1-stability})}{ \lesssim } 
t h_T^{-1} \|u_0\|_{H^{m-1}(\omega^2 (T))}. 
\end{align*}
Summation over all elements $T \in \TT_{>t}$ and 
Corollary~\ref{cor:H12-estimate-scott-zhang}, \eqref{eq:cor:H12-estimate-scott-zhang-20}--\eqref{eq:cor:H12-estimate-scott-zhang-25} 
(noting that $t<h_T$ implies $\overline{h} = h$ on $\TT_{>t}$) give
the desired estimate
\begin{align}\label{eq:thm:stability-discontinuous-tmp1}
&\sum_{T \in \TT_{>t}}  \|\II u_0 - \average{\delta t}{(\II u_0)}\|_{H^{m-1}(T \setminus S_{T,\delta t})}^2 \lesssim 
 t^2\sum_{T \in \TT_{>t}} h_T^{-2} \|u_0\|_{H^{m-1}(\omega^2 (T))}^2 \nonumber\\ 
 &\qquad\qquad\qquad\stackrel{t<h_T}\lesssim t \sum_{j=0}^{m-1}
 \norm{\overline{h}^{-1/2} \nabla^j(u-J_t u) }_{L^2(\Omega)}^2 \lesssim t\norm{u}_{H^{m-1/2}(\Omega)}^2.
\end{align}
Similarly, the approximation properties of $\average{\delta t}$, the stability of $\II$, and 
Corollary~\ref{cor:H12-estimate-scott-zhang} give 
\begin{align}\label{eq:thm:stability-discontinuous-tmp2}
\nonumber
 \sum_{T \in \TT_{>t}} \|\II u_1 - \average{\delta t}{(\II u_1)}\|_{H^{m-1}(T \setminus S_{T,\delta t})}^2 
 &\lesssim t^2 \sum_{T \in \TT_{>t}} \|\II u_1\|_{H^m(T)}^2  
  \stackrel{(\ref{eq:discontinuous-H1-stability})}{ \lesssim } 
  t^2 \sum_{T \in \TT_{>t}} \|u_1\|_{H^m(\omega^2 (T))}^2 \\
&\lesssim t \sum_{T \in \TT_{>t}}h_T \|J_t u\|_{H^m(\omega^2 (T))}^2 \lesssim t\norm{u}_{H^{m-1/2}(\Omega)}^2. 
\end{align}
Using the stability instead of the approximation properties of 
$\average{\delta t}$ from Proposition~\ref{prop:KM15}, the same arguments and an inverse estimate lead to  
\begin{align*}
t\|\average{\delta t}{(\II u_0)}\|_{H^m(T\setminus S_{T,\delta t})} 
&\lesssim t\|\II u_0\|_{H^m(T)} 
\lesssim t h_T^{-1} \|u_0\|_{H^{m-1}(\omega^2 (T))} 
\end{align*}
as well as
\begin{align*}
t\|\average{\delta t}{(\II u_1)}\|_{H^m(T\setminus S_{T,\delta t})} 
&\lesssim t\|\II u_1\|_{H^m(T)} 
\lesssim  t\|u_1\|_{H^m(\omega^2 (T))}. 
\end{align*}
Summing and employing Corollary~\ref{cor:H12-estimate-scott-zhang}, we obtain the 
desired estimates as in \eqref{eq:thm:stability-discontinuous-tmp1} and \eqref{eq:thm:stability-discontinuous-tmp2}.
\newline

{\em 5.~step:} We derive approximation results for $\II$ on the strip $S_{T,\delta t}$ for $T \in \TT_{>t}$. 
For $v \in H^m(\Omega)$, we claim
\begin{align}
\label{eq:thm:stability-discontinuous-185}
\|v - \II v\|_{H^{m-1}(S_{T,\delta  t})} 
&\lesssim  \sqrt{t h_T} \|v\|_{H^m(\omega^2(T))}. 
\end{align}
With the aid of \cite[Lemma~{2.1}]{li-melenk-wohlmuth-zou10} on the reference element and a scaling argument, 
one can show for $v \in H^1(T)$ and $T \in {\mathcal T}_{> t}$
\begin{align}
\label{eq:thm:stability-discontinuous-170}
\|v\|^2_{L^2(S_{T,\delta t})} &\lesssim \frac{t}{h_T} \|v\|^2_{L^2(T)} + t \|v\|_{L^2(T)} \|\nabla v\|_{L^2(T)}. 
\end{align} 
For polynomials $v \in P_p(T)$, an inverse estimate 
and \eqref{eq:thm:stability-discontinuous-170} furthermore lead to 
\begin{align}
\label{eq:thm:stability-discontinuous-180}
\|v\|^2_{L^2(S_{T,\delta t})} &\lesssim \frac{t}{h_T} \|v\|^2_{L^2(T)}. 
\end{align}
To see \eqref{eq:thm:stability-discontinuous-185}, we estimate 
\begin{align*}
\|v - \II v\|^2_{L^2(S_{T,\delta  t})} &
\stackrel{\eqref{eq:thm:stability-discontinuous-170}} {\lesssim}
\frac{t}{h_T}\|v - \II v\|^2_{L^2(T)}  
+ t \|v - \II v\|_{L^2(T)} \|\nabla( v - \II v)\|_{L^2(T)}  \\
& \stackrel{\eqref{eq:lemma:approximation-10}}{\lesssim }
h_T t \|v\|^2_{H^1(\omega^2(T))}.
\end{align*}
Applying \eqref{eq:thm:stability-discontinuous-170} to derivative of $u-\II u$ for $m=2$, the same argument gives 
\begin{align*}
\|v - \II v\|^2_{H^{m-1}(S_{T,\delta  t})} & \lesssim 
h_T t \|v\|^2_{H^m(\omega^2(T))}. 
\end{align*}

{\em 6.~step:} We derive estimate of
$\II u_i - \average{\delta t}{(\II u_i)}$ on the strip $S_{T,\delta t}$ for $T \in \TT_{>t}$.
Here, we need the ``$\delta t$-neighborhood'' $(S_{T,\delta t})_{\delta t}$ of the strip $S_{T,\delta t}$.
Our assumption on $\delta$ implies that $(S_{T,\delta t})_{\delta t} \subset \omega(T)$. Moreover,
we note that the strip $(S_{T,\delta t})_{\delta t}$ is contained in the inside strip $S_{T,2\delta t}$  
of $T$ and in parts of the inside strip of width $\delta t$ of the elements $T' \in \omega(T)$.

Using Proposition~\ref{prop:KM15} and 
\eqref{eq:thm:stability-discontinuous-180} on each element of the patch $\omega(T)$ separately
for $v = \II u_0$ in the case $m = 1$ or 
$v = \nabla \II u_0$ for $m=2$, we get, since $h_{T'} \sim h_T$ for $T' \in \omega(T)$, 
\begin{align}\label{eq:eststriptmp1}
\nonumber 
\|\II u_0 - \average{\delta t}{(\II u_0)}\|_{H^{m-1}(S_{T,\delta t})} &\leq 
\|\II u_0\|_{H^{m-1}(S_{T,\delta t})} + \|\average{\delta t}{(\II u_0)}\|_{H^{m-1}(S_{T,\delta t})} \\
&\lesssim 
\|\II u_0\|_{H^{m-1}((S_{T,\delta t})_{\delta t})} 
\stackrel{\eqref{eq:thm:stability-discontinuous-180}}
{\lesssim} t^{1/2}h_T^{-1/2}
\|\II u_0\|_{H^{m-1}(\omega (T))}  \\ \nonumber
&\lesssim t^{1/2}h_T^{-1/2} \| u_0\|_{H^{m-1}(\omega^3(T))}.
\end{align}
Summing over all elements $T \in \TT_{>t}$ and employing the argument from \eqref{eq:thm:stability-discontinuous-tmp1},
we get the desired bound by $t^{1/2} \norm{u}_{H^{m-1/2}(\Omega)}$.
For $u_1$, we use the triangle inequality, Proposition~\ref{prop:KM15}, and \eqref{eq:thm:stability-discontinuous-185}
\begin{align*}
& \|\II u_1 - \average{\delta t}{(\II u_1)}\|_{H^{m-1}(S_{T,\delta t})}  \\
&\qquad\leq
\|\II u_1 - u_1\|_{H^{m-1}(S_{T,\delta t})} + \|u_1 - \average{\delta t}{u_1}\|_{H^{m-1}(S_{T,\delta t})} 
+ \|\average{\delta t}{(u_1-\II u_1)}\|_{H^{m-1}(S_{T,\delta t})} \\
&\qquad \stackrel{\text{Prop.~\ref{prop:KM15}}}{\lesssim}  
\|\II u_1 - u_1\|_{H^{m-1}((S_{T,\delta t})_{\delta t})} + \|u_1 - \average{\delta t}{u_1}\|_{H^{m-1}(S_{T,\delta t})} 
\\
&\qquad \stackrel{
\eqref{eq:thm:stability-discontinuous-185}, 
\text{Prop.~\ref{prop:KM15}}}{\lesssim}  
\sqrt{ t h_T} \|u_1\|_{H^m(\omega^3(T))} + t \|u_1\|_{H^m(\omega(T))} 
\stackrel{t \leq h_T}{\lesssim} \sqrt{t h_T}\|u_1\|_{H^m(\omega^3(T))}. 
\end{align*}
Summing over all elements $T \in \TT_{>t}$ and employing the argument from \eqref{eq:thm:stability-discontinuous-tmp2},
we get the desired bound by $t^{1/2} \norm{u}_{H^{m-1/2}(\Omega)}$.\\

{\em 7.~step:} We estimate $\average{\delta t}{(\II u_i)}$ on the strip $S_{T,\delta t}$ for $T \in \TT_{>t}$.
The inverse estimate for $\average{\delta t}{}$ of Proposition~\ref{prop:KM15}, 
\eqref{eq:thm:stability-discontinuous-180} employed on the patch $\omega(T)$ as in the previous step, and 
the stability (\ref{eq:discontinuous-H1-stability}) of $\II$ imply 
\begin{align}\label{eq:eststriptmp2}
t\|\average{\delta t}{(\II u_0)}\|_{H^m(S_{T,\delta t})} &\lesssim 
\|\II u_0\|_{H^{m-1}((S_{T,\delta t})_{\delta t})} \lesssim 
t^{1/2}h_T^{-1/2} \|\II u_0\|_{H^{m-1}(\omega (T))} \\ \nonumber
&\lesssim 
t^{1/2}h_T^{-1/2} \|u_0\|_{H^{m-1}(\omega^3 (T))}. 
\end{align}
Summing over all elements $T \in \TT_{>t}$ and employing the argument from \eqref{eq:thm:stability-discontinuous-tmp1},
we get the desired bound by $t^{1/2} \norm{u}_{H^{m-1/2}(\Omega)}$.
For $u_1$, Proposition~\ref{prop:KM15} and \eqref{eq:thm:stability-discontinuous-185} on the patch
$\omega(T)$ give
\begin{align*}
t\|\average{\delta t}{(\II u_1)}\|_{H^m(S_{T,\delta t})} &\leq 
t\|\average{\delta t}{(u_1 - \II u_1)}\|_{H^m(S_{T,\delta t})} + 
t\|\average{\delta t}{u_1}\|_{H^m(S_{T,\delta t})} 
\\
&\lesssim \|u_1 - \II u_1\|_{H^{m-1}((S_{T,\delta t})_{\delta t})} + t\|u_1\|_{H^m((S_{T,\delta t})_{\delta t})}  \\
&\stackrel{(\ref{eq:thm:stability-discontinuous-185})}{\lesssim} 
(t h_T)^{1/2} \|u_1\|_{H^m(\omega^3 (T))} + t\|u_1\|_{H^m(\omega^3 (T))} 
\\
&\stackrel{t < h_T}{\lesssim}  (t h_T)^{1/2} \|u_1\|_{H^m(\omega^3 (T))}.
\end{align*}
Summing over all elements $T \in \TT_{>t}$ and employing the argument from \eqref{eq:thm:stability-discontinuous-tmp2},
we get the desired bound by $t^{1/2} \norm{u}_{H^{m-1/2}(\Omega)}$. \newline

Combining the estimates of steps 2--7, where all relevant terms are bounded by $t^{1/2} \norm{u}_{H^{m-1/2}(\Omega)}$,
gives the desired bound for \eqref{eq:KfunctionalEstimate}, which proves \eqref{eq:thm:stability-discontinuous-i}.
\newline

{\em Final step:} We show (\ref{eq:thm:stability-discontinuous-ii}) with similar arguments as in step 2--7. 
Let $u = u_0 + u_1$ be an arbitrary decomposition with $u_0 \in H^{m-1}(\Omega)$ and $u_1 \in H^m(\Omega)$. 
We distinguish the cases $t \leq h$ and 
$t > h$, where $h$ is the maximal mesh size of the quasi-uniform triangulation. We note that in the decomposition 
\eqref{eq:domainDecomp} the sums $\sum_{T \in {\mathcal T}_{> t}}$ are not present in the case 
$t > h$ and the terms involving $\|\cdot\|_{H^{m-1}(\Omega_{\leq t})}$ or $\|\cdot \|_{H^m(\Omega_{\leq t})}$ in the converse case. 
Inspection of the above arguments therefore gives: 
\begin{itemize}
\item For $t > h$: As in steps 2--3, we get
\begin{align*}
t^{-1} \|\II u_0  - \average{\delta t}{(\II  u_0)}\|^2_{H^{m-1}(\Omega)} + t \|\average{\delta t}{(\II  u_0)}\|^2_{H^m(\Omega)} 
&\lesssim t^{-1} \|u_0\|^2_{H^{m-1}(\Omega)}, \\
t^{-1} \|\II u_1  - \average{\delta t}{(\II  u_1)}\|^2_{H^{m-1}(\Omega)} + t \|\average{\delta t}{(\II  u_1)}\|^2_{H^m(\Omega)} 
&\lesssim t \|u_1\|^2_{H^{m}(\Omega)}. 
\end{align*}
This implies $t^{-1/2} K(t,\II u) \lesssim t^{-1/2} \|u_0\|_{H^{m-1}(\Omega)} + t^{1/2} \|u_1\|_{H^m(\Omega)}$. Infimizing over all possible 
decompositions $u = u_0 + u_1$ yields $t^{-1/2} K(t,\II u) \lesssim t^{-1/2} K(t,u) \lesssim \|u\|_{B^{m-1/2}_{2,\infty}(\Omega)}$. 
\item For $t \leq h$: As in steps 4--7, we get
\begin{align*}
t^{-1} \|\II u_0  - \average{\delta t}{(\II  u_0)}\|^2_{H^{m-1}(\Omega)} + t \|\average{\delta t}{(\II  u_0)}\|^2_{H^m(\Omega)} 
&\lesssim h^{-1} \|u_0\|^2_{H^{m-1}(\Omega)}, \\
t^{-1} \|\II u_1  - \average{\delta t}{(\II  u_1)}\|^2_{H^{m-1}(\Omega)} + t \|\average{\delta t}{(\II  u_1)}\|^2_{H^m(\Omega)} 
&\lesssim h \|u_1\|^2_{H^{m}(\Omega)}. 
\end{align*}
This implies $t^{-1/2} K(t,\II u) \lesssim h^{-1/2} \|u_0\|_{H^{m-1}(\Omega)} + h^{1/2} \|u_1\|_{H^m(\Omega)}$. Infimizing over all possible 
decompositions $u = u_0 + u_1$ yields $t^{-1/2} K(t,\II u) \lesssim h^{-1/2} K(h,u) \lesssim \|u\|_{B^{m-1/2}_{2,\infty}(\Omega)}$. 
\end{itemize}
Combining the above two cases yields $\sup_{t > 0} K(t,\II u) \lesssim \|u\|_{B^{m-1/2}_{2,\infty}(\Omega)}$, as claimed. 
\end{proof}

While for \emph{finite} meshes we have the continuous embeddings 
$S^{p,1}(\TT) \subset B^{3/2}_{2,\infty}(\Omega)$
and $S^{p,0}(\TT) \subset B^{1/2}_{2,\infty}(\Omega)$, this is not necessarily
the case for infinite meshes. As a consequence, one cannot expect that 
on general K-meshes a stability 
$\II:B^{1/2}_{2,\infty}(\Omega) \rightarrow B^{1/2}_{2,\infty}(\Omega)$ can
hold. The following example illustrates this.

\begin{example}
Let $\Omega = (0,1)$. Set $I_1 = (0,1/2)$ and $I_2 = (1/2,1)$.  
Let $\varphi \in C^\infty(\BbbR)$ be a $1$-periodic function whose averages 
$\overline{\varphi}_1:= 1/|I_1| \int_{I_1} \varphi(x)\,dx$ and 
$\overline{\varphi}_2:= 1/|I_2| \int_{I_2} \varphi(x)\,dx$ are {\em different}. 
Define the function $u \in C^\infty((0,\infty))$ by 
$$
u(x):= \varphi(\ln x). 
$$
Define the (infinite) mesh ${\mathcal T}$ on $\Omega$ whose elements 
are given by the break points $x_j = e^{-2j}$, $j\in \N_0$. 
Let $m = 1$ and let $\II:L^2(\Omega) \rightarrow S^{0,0}({\mathcal T})$ be the 
$L^2$-projection onto the piecewise constant functions. By the periodicity 
of $\varphi$ the piecewise 
constant function $\II u$ takes only the values $\overline{\varphi}_{1}$ and $\overline{\varphi}_2$
\begin{align*}
(\II u)|_{(x_{j+1},x_{j})} = 
\begin{cases}
\overline{\varphi}_1 & \mbox{ if $j$ is even } \\
\overline{\varphi}_2 & \mbox{ if $j$ is odd }.  
\end{cases}
\end{align*}
The computation of Besov norms is conveniently done in terms of the modulus of smoothness as defined 
in, e.g., \cite[Chap.~2, Sec.~7]{devore93}. 
For an interval $[a,b]$ and a function $v$ defined on $A:=[a,b]$, and $t > 0$, we define
the difference operator $\Delta_1$ by $(\Delta_h v)(x) := v(x+h) - v(x)$ on $A_h:= [a,b-h]$.  
the modulus of smoothness $\omega_1(v,t)_2$ is then given by 
$\omega_1(v,t)_2:= \sup_{0 < h \leq t} \|\Delta_h(v,\cdot)\|_{L^2(A_{h})}$. 
Let $t > 0$. Consider all elements with diameter $ > t$. For the region covered by these elements, 
$\Omega_{> t}$, we can compute the modulus of smoothness $\omega_1$ 
in view of the fact that $\II u$ is piecewise constant 
\begin{align*}
\omega_1(\II u,t)_{2,\Omega_{>t}}^2 = \sum_{x_j: x_j > t} t |[\II u]({x_j})|^2, 
\end{align*}
where $[\II u]({x_j})$ denotes the jump of $\II u$ at the break point $x_j$. We conclude 
\begin{align*}
\omega_1(\II u,t)^2_2 \ge \omega_1(\II u,t)_{2,\Omega_{>t}}^2 = \sum_{x_j: x_j > t} t |[\II u]({x_j})|^2 
 = \sum_{x_j: x_j > t} |\overline{\varphi}_1 - \overline{\varphi}_2|^2 t  \sim 
 |\overline{\varphi}_1 - \overline{\varphi}_2|^2 t |\ln t|. 
\end{align*}
Next, we claim that $\omega_1(u,t)^2_2 \lesssim t$. Since $u$ is bounded, we compute for $0 < h \leq t$
we estimate 
\begin{align*} 
\int_{0}^{1-h} |\Delta_h u|^2\,dx & =
\int_{0}^{1-h} |u(x+h) - u(x)|^2\,dx \\
& = 
\int_{0}^{h} |u(x+h) - u(x)|^2\,dx  + 
\int_{t}^{1-h} |u(x+h) - u(x)|^2\,dx   \\
& \leq 2 h \|u\|^2 _{L^\infty(\Omega)} + \int_{h}^1 \left| \int_{x}^{x+h} u^\prime(\xi)\,d\xi\right|^2\,dx \\
& \leq 2 h \|u\|^2 _{L^\infty(\Omega)} + \int_{h}^1 \left| \int_{x}^{x+h} \|\varphi^\prime\|_{L^\infty(\Omega)} \frac{1}{\xi} \,d\xi\right|^2\,dx \\ 
& \leq 2 h \|u\|^2 _{L^\infty(\Omega)}  + \|\varphi^\prime\|^2_{L^\infty(\Omega)} \int_{h}^1 \left(\frac{h}{x}\right)^2 \,dx 
\leq 2 h \|u\|^2 _{L^\infty(\Omega)} + \|\varphi^\prime\|^2_{L^\infty(\Omega)} h. 
\end{align*}
This implies $\omega_1(u,t)_2 \leq C t^{1/2}$ and therefore 
$u \in B^{1/2}_{2,\infty}(\Omega)$, since, by  
\cite[Chap.~6, Thm.~{2.4}]{devore93}, $\omega(u,t)_2 \sim K(t,u) = 
\inf_{v \in H^1(I)} \|u - v\|_{L^2(\Omega)} + t\|v\|_{H^1(\Omega)}$. 
However, the above calculation shows that 
$\II u \not\in B^{1/2}_{2,\infty}(\Omega)$, which implies that 
we will not be able to control $\II u$ uniformly in the mesh.
\eremk
\end{example}

\subsection{Some generalizations and applications}

For quasi-uniform meshes, there also holds the following inverse estimate for the limiting case. 
\begin{lemma}
\label{lemma:inverse-estimate}
Let $\TT$ be a quasi-uniform mesh on $\Omega$ of mesh size $h$ and $ m \in \{1,2\}$. Then, 
\begin{align}
\label{eq:lemma:inverse-estimate-10}
\|u\|_{B^{m-1/2}_{2,\infty}(\Omega)
      } \leq C h^{-(m-1/2)}  \|u\|_{L^2(\Omega)} 
\qquad \forall u \in S^{p,1}(\TT). 
\end{align}
More generally, for any real $m' \in (0,m-1/2)$ and $q \in [1,\infty]$, there holds by 
interpolation
\begin{align}
\label{eq:lemma:inverse-estimate-15}
\|u\|_{B^{m'}_{2,q}(\Omega)} \leq C h^{-m'}  \|u\|_{L^2(\Omega)} 
\qquad \forall u \in S^{p,1}(\TT). 
\end{align}
The constant $C > 0$ depends only on $\Omega,d$, the $\gamma$-shape-regularity of $\TT$, and $p$. 
\end{lemma}
\begin{proof}
To fix ideas, we only prove
the case $m = 2$ as the case $m = 1$ is handled with similar arguments. By definition, we have
\begin{align*}
\|u\|_{B^{3/2}_{2,\infty}(\Omega)} & = 
\sup_{t > 0} t^{-1/2} K(t,u) 
\end{align*}
with the $K$-functional $K(t,u) = \inf_{v \in H^2(\Omega)} \|u - v\|_{H^1(\Omega)} + t \|v\|_{H^2(\Omega)}$. 
For $t > h$, we estimate 
\begin{align}
\label{eq:lemma:inverse-estimate-20}
t^{-1/2} K(t,u) & = t^{-1/2} \inf_{v \in H^2(\Omega)} \|u - v\|_{H^1(\Omega)} + t \|v\|_{H^2(\Omega)}
\leq t^{-1/2} \|u\|_{H^1(\Omega)} 
\nonumber \\ &
\lesssim h^{-1/2} \|u\|_{H^ {1}(\Omega)} 
\end{align}
by choosing $v\equiv 0$ to estimate the $K$-functional. 

For $t \leq h$, we estimate the $K$-functional more carefully. For a suitably small $\delta>0$,
we set  $v:= {\mathcal A}_{\delta t} u$ with the smoothing operator ${\mathcal A}_{\delta t}$
of Proposition~\ref{prop:KM15}. 
As in the proof of Theorem~\ref{thm:stability-discontinuous}, we decompose an element into 
$T = T\setminus S_{T, \delta t} \cup S_{T, \delta t}$, where $S_{T, \delta t}$ is the inside strip 
defined in the first step of the proof of Theorem~\ref{thm:stability-discontinuous}.
Employing Proposition~\ref{prop:KM15} and a classical polynomial inverse estimate, we obtain
\begin{subequations}
\label{eq:lemma:inverse-estimate-30}
\begin{align}
\|v\|_{H^2(T\setminus S_{T, \delta t})} &\stackrel{\text{Prop.~\ref{prop:KM15}}}{\lesssim} 
\|u\|_{H^2(T)} \lesssim h^{-1} \|u\|_{H^1(T)}, \\
\|u - v\|_{H^1(T\setminus S_{T,\delta t})} &
\stackrel{\text{Prop.~\ref{prop:KM15}}}{\lesssim } t \|u\|_{H^2(T)}
\lesssim t h^{-1} \|u\|_{H^1(T)}.
\end{align}
\end{subequations}
As in steps 6--7 in the proof of Theorem~\ref{thm:stability-discontinuous}, using Proposition~\ref{prop:KM15}
to obtain \eqref{eq:eststriptmp2}, \eqref{eq:eststriptmp1}, we get
\begin{subequations}
\label{eq:lemma:inverse-estimate-35}
\begin{align}
\|v\|_{H^2(S_{T, \delta t})} & 
\stackrel{(\ref{eq:eststriptmp2})}{\lesssim} (th)^{-1/2}  \|u\|_{H^1(\omega(T))}, \\ 
\|u - v\|_{H^1(S_{T,\delta t})} &
\stackrel{(\ref{eq:eststriptmp1})}{\lesssim} t^{1/2} h^{-1/2}  \|u\|_{H^1(\omega(T))}.  
\end{align}
\end{subequations}
Summation over all elements, using (\ref{eq:lemma:inverse-estimate-30})--(\ref{eq:lemma:inverse-estimate-35}) leads to
\begin{align}
\label{eq:lemma:inverse-estimate-40}
t^{-1/2} K(t,u) &\lesssim \left(t^{1/2} h^{-1} + h^{-1/2} \right)\|u\|_{H^1(\Omega)} 
\stackrel{t \leq h}{\lesssim } h^{-1/2} \|u\|_{H^1(\Omega)}. 
\end{align}
Combining (\ref{eq:lemma:inverse-estimate-20}) and (\ref{eq:lemma:inverse-estimate-40}) yields 
$\|u\|_{B^{3/2}_{2,\infty}(\Omega)} \lesssim h^{-1/2} \|u\|_{H^1(\Omega)}$. A further polynomial inverse estimate
gives the desired result.  \\

Finally, \eqref{eq:lemma:inverse-estimate-15} follows from interpolation between 
\eqref{eq:lemma:inverse-estimate-10} and the trivial inequality $\|u\|_{L^2(\Omega)} \leq \|u\|_{L^2(\Omega)}$
noting that 
by the reinterpolation theorem (see, e.g., 
\cite[Chap.~{26}]{tartar07}), we have 
$B^{\theta (m-1/2)}_{2,q}(\Omega) = (L^2(\Omega), B^{m-1/2}_{2,\infty}(\Omega))_{\theta,q}$ (with equivalent norms)
for $\theta \in (0,1)$. 
\end{proof}

The operator $\II$ is stable in $L^2(\Omega)$ (by Assumption~\ref{assumption:I-discontinuous})  
and is stable as an operator $H^{m-1/2}(\Omega) \rightarrow B^{m-1/2}_{2,\infty}(\Omega)$
by Theorem~\ref{thm:stability-discontinuous}. Interpolation therefore yields a stability 
for intermediate spaces.

\begin{corollary}
\label{cor:stability}
Let $\TT$ be a finite shape-regular mesh, $m \in \{1,2\}$, and let $\II:L^2(\Omega) \rightarrow S^{p,m-1}(\TT)$ satisfy
Assumption~\ref{assumption:I-discontinuous}. Fix $q \in [1,\infty]$ and $\theta \in (0,1)$. 
Then, there is a constant $C > 0$ depending only on $\Omega$, $p$, $q$, $\theta$, and the $\gamma$-shape regularity of $\TT$
such that 
\begin{equation}
\label{eq:cor-stability-10}
\|\II u\|_{B^{\theta (m-1/2)}_{2,q}(\Omega)} \leq C \|u\|_{B^{\theta (m-1/2)}_{2,q}(\Omega)}.  
\end{equation}
\end{corollary}
\begin{proof} 
The assumed $L^2$-stability and the stability proved in Theorem~\ref{thm:stability-discontinuous} 
imply the result using the reinterpolation theorem (see, e.g., 
\cite[Chap.~{26}]{tartar07}) to observe 
$B^{\theta (m-1/2)}_{2,q}(\Omega) = (L^2(\Omega), B^{m-1/2}_{2,\infty}(\Omega))_{\theta,q}$ (with equivalent norms). 
\end{proof}

Furthermore, Corollary~\ref{cor:stability} allows one to assert that the interpolating 
between the discrete space $S^{p,m-1}(\TT)$ equipped with the $L^2$-norm and the $H^s$-norm yields
the same space equipped with the $H^{s\theta}$-norm.

\begin{corollary}
\label{cor:interpolation} 
Let $m \in \{1,2\}$, $q \in [1,\infty]$, and $\theta \in (0,1)$. Then, there holds 
\begin{align*}
\left( (S^{p,m-1}(\TT), \|\cdot\|_{L^2(\Omega)}),  
       (S^{p,m-1}(\TT), \|\cdot\|_{B^{m-1/2}_{2,\infty}(\Omega)})
\right)_{\theta,q} = 
       (S^{p,m-1}(\TT), \|\cdot\|_{B^{\theta( m-1/2)}_{2,q}(\Omega)})
\end{align*}
with equivalent norms. The norm equivalence constants depend only on $\Omega$, $p$, $q$, $\theta$, 
and the $\gamma$-shape regularity of $\TT$. 
More generally, for any $B^{m'-1/2}_{2,q'}(\Omega)$ with $1/2 < m' < m$ and $q' \in [1,\infty]$, there holds, 
with equivalent norms,   
\begin{align*}
\left( (S^{p,m-1}(\TT), \|\cdot\|_{L^2(\Omega)}),  
       (S^{p,m-1}(\TT), \|\cdot\|_{B^{m'-1/2}_{2,q'}(\Omega)})
\right)_{\theta,q} = 
       (S^{p,m-1}(\TT), \|\cdot\|_{B^{\theta(m'-1/2)}_{2,q}(\Omega)}). 
\end{align*}
\end{corollary} 
\begin{proof}
The proof follows from the existence of projection operators as presented in
\cite{arioli-loghin09}. 
One needs a (stable) projection onto $S^{p,m-1}(\TT)$ satisfying Assumption~\ref{assumption:I-discontinuous},
then Corollary~\ref{cor:stability} also provides the needed stability in the Besov-spaces.

For $m = 1$, 
one may simply use the $L^2$-projection, which trivially satisfies Assumption~\ref{assumption:I-discontinuous}. 
For $m=2$, one employs the Scott-Zhang operator $I^{SZ}$ of \cite{scott-zhang90} without treating the boundary 
in a special way as it is done there. Then, $I^{SZ}$ satisfies Assumption~\ref{assumption:I-discontinuous} 
by, e.g., \cite[Sec.~4.8]{brenner-scott02}. 
\end{proof}

\section{Multilevel decomposition based on NVB mesh-hierarchy} 
\label{sec:NVBproof}

In this section, we use Theorem~\ref{thm:stability-discontinuous}, or more precisely 
Corollary~\ref{cor:interpolation}, to prove the norm equivalence for the multilevel decomposition
of Theorem~\ref{thm:multilevel-fcc}.
Before we come to the proof, we mention some properties of the finest common coarsening and show that 
the adapted Scott-Zhang operators of Definition~\ref{def:adapted-SZ} for the finest common coarsening 
of an NVB mesh and a uniform mesh coincides with the adapted Scott-Zhang operator for the uniform mesh 
for piecewise polynomials on the NVB mesh.

\subsection{Properties of the finest common coarsening ($\fcc$)}

We recall the definition of the finest common coarsening
\begin{align*}
& \fcc(\TT,\TT^\prime) := \\
\nonumber 
& \qquad \underbrace{ \{T \in \TT\,:\, \exists T' \in \TT^\prime \mbox{ s.t. }  T' \subsetneq T\}}_{=:\TB_1}
\cup \underbrace{ \{T^\prime \in \TT^\prime\,:\, \exists T \in \TT \mbox{ s.t. }  T \subsetneq T^\prime\}}_{=:\TB_2} 
\cup \underbrace{ (\TT \cap \TT^\prime) }_{=:\TB_3}. 
\end{align*}

The following Lemma~\ref{lemma:fcc} shows that the finest common coarsening of two NVB meshes obtained from the same 
coarse regular triangulation is indeed a regular triangulation. 
\begin{lemma}
\label{lemma:fcc}
Let $\TT$, $\TT^\prime$ be NVB-refinements of the same common
triangulation $\widehat\TT_0$ of $\Omega$. Then: 
\begin{enumerate}[(i)]
\item
\label{item:lemma:fcc-0}
$\fcc(\TT,\TT^\prime) = 
\fcc(\TT^\prime,\TT)$. 
The three sets $\TB_1$, $\TB_2$, $\TB_3$ in the definition of $\fcc(\TT,\TT^\prime)$ are pairwise disjoint. 
\item 
\label{item:lemma:fcc-i}
$\fcc(\TT,\TT^\prime)$ consists of simplices that cover $\Omega$. 
\item
\label{item:lemma:fcc-ii}
If $\TT$ and $\TT^\prime$ are regular triangulations, then 
$\fcc(\TT,\TT^\prime)$ is a regular triangulation of $\Omega$. 
\end{enumerate}
\end{lemma}
\begin{proof}
\emph{Proof of (\ref{item:lemma:fcc-0}):} The symmetry of $\fcc$ is obvious. 
To see that the sets $\TB_1$, $\TB_2$, $\TB_3$ are pairwise disjoint, let 
$T \in \TB_1$. Then $T \in \TT$ but not in $\TT^\prime$. Hence, 
$T \not\in \TB_2$ and $T \not\in \TB_3$. By symmetry, $T \in \TB_2$ also 
implies $T \not\in\TB_1$ and $T \not\in \TB_3$. Finally, if 
$T \in \TB_3$, then it cannot be in $\TB_1$ or $\TB_2$. 

\emph{Proof of (\ref{item:lemma:fcc-i}):}
Let $x \in \Omega$ (but not on the skeleton of $\TT$ or $\TT^\prime$). 
Since $\TT$, $\TT^\prime$ cover $\Omega$, there are $T \in \TT$ and 
$T^\prime \in \TT^\prime$ with $x \in T$, $x \in T^\prime$. 
Since both $T$ and $T^\prime$ are obtained by NVB and $T \cap T^\prime \ne \emptyset$, we must
have $T = T^\prime$ or $T \subsetneq T^\prime$ or $T^\prime  \subsetneq T$. In the first 
case $T = T^\prime \in \TB_3$, in the second one $T^\prime \in \TB_2$, and in the third
one $T \in \TB_1$. Hence, $x$ is in an element of $\fcc(\TT,\TT^\prime)$. 

\emph{Proof of (\ref{item:lemma:fcc-ii}):}
Let $T$, $T^\prime$ be two elements of $\fcc(\TT,\TT^\prime)$ with 
$f:= \overline{T} \cap \overline{T^\prime} \ne \emptyset$. We have to show 
that for some $j$, the intersection 
$\overline{T} \cap \overline{T^\prime} \ne \emptyset$ is a full $j$-face 
of both $T$ and $T^\prime$. If both $T$, $T^\prime$ are in $\TT$ (or both are in $\TT^\prime$),
then, by the regularity of $\TT$ (or the regularity of $\TT^\prime$), their intersection is 
indeed a \emph{full} $j$-face of either element. 
Assume therefore $T\in \TT$ and $T^\prime \in \TT^\prime$. 
Since $T$, $T^\prime \in \fcc(\TT,\TT^\prime)$, we obtain $T \in \TB_1$ and $T^\prime \in \TB_2$. 
Since both $T$ and $T^\prime$ are created by NVB from the same initial triangulation, 
the intersection $f = \overline{T} \cap \overline{T^\prime}$ is a full $j$-face of either $T$ or $T^\prime$. 

Let us assume that $f$ is a full $j$-face of $T$, and, by contradiction,  
that $f$ is not a full $j$-face of $T^\prime$. 
Then, $f$ is a proper subset of a $j$-face $f^\prime$ of $T^\prime$. Since $T \in \TB_1$, it contains 
elements of $\TT^\prime$. Hence, there is an element $T^\prime_1 \in \TT^\prime$ with 
$T^\prime_1 \subset T$ that has a $j$-face $f_1^\prime$ with $f_1^\prime \subset f$. Thus, 
we have found elements $T^\prime$, $T^\prime_1 \in \TT^\prime$ with $j$-faces 
$f_1^\prime \subset f \subsetneq f^\prime$, contradicting the regularity of $\TT^\prime$. 
Hence, $f$ is also a full $j$-face of $T^\prime$. Thus, $\fcc(\TT,\TT^\prime)$ is a regular triangulation. 

%
%
\end{proof}

A completion of an (NVB-generated) mesh is any NVB-refinement of it that is regular. 
We next show that the minimal completion is unique. 
\begin{lemma}
\label{lemma:minimal-complection-unique}
Let $\TT$ be a NVB-refinement of $\widehat \TT_0$ and let $\TT_1$, $\TT_2$ be two 
completions of $\TT$. Then $\fcc(\TT_1, \TT_2)$ is a completion of 
$\TT$. The completion of minimal cardinality is unique. 
\end{lemma}
\begin{proof}
Let $\TT_3:= \fcc(\TT_1,\TT_2)$. 
We claim that $\TT_3$ is a completion of $\TT$. Since $\TT_3$ is regular
by Lemma~\ref{lemma:fcc}, we have to assert that each element of $\TT_3$ is 
contained in an element of $\TT$. Suppose not. Then there is $T_3 \in \TT_3$ 
and a $T \in \TT$ with $T \subsetneq T_3$. (We use that these meshes are obtained
by NVB from a common $\TT_0$.). By definition, $T_3$ is either in 
$\TT_1$ or $\TT_2$, which are both completions of $\TT$, i.e., their elements
are contained in elements of $\TT$. This is a contradiction. 

To see the uniqueness of the minimal completion, let $\TT_1 \ne \TT_2$ be 
two completions of minimal cardinality $N$. Note that $\TT_3:= \fcc(\TT_1,\TT_2)$
is also a completion. However, in view of $\TT_1 \ne \TT_2$, at least one element
of $\TT_1$ is a refinement of an element of $\TT_2$ so that 
we have by definition of $\fcc(\TT_1, \TT_2)$ that 
$\operatorname{card} \TT_3 \leq N-1$, which contradicts the minimality.  
\end{proof}

\begin{lemma} 
\label{lemma:chopping}
Let $\widehat\TT_\ell$, $\ell=0,1,\ldots,$ be a sequence of uniform refinements of a regular mesh $\widehat \TT_0$
and $\widetilde \TT_\ell = \fcc(\TT, \widehat\TT_\ell)$. Then:
\begin{enumerate}[(i)]
\item 
\label{item:lemma:chopping-ii}
If $T \in \Tt_\ell \cap \TT$ then $T \in \Tt_{\ell+m}$ for all $m \ge 0$. 
\item 
\label{item:lemma:chopping-iii}
If $T \in \Tt_\ell\setminus \TT$ then $T \not\in \Tt_{\ell+1}$.
\item 
\label{item:lemma:chopping-iv}
Denote by $\widetilde{\mathcal N}_\ell^1$ the set of nodes of $\widetilde\TT_\ell$. 
Then $\widetilde{\mathcal N}_{\ell+1}^1 \supset \widetilde {\mathcal N}_{\ell}^1$ for all $\ell$. 
\item 
\label{item:lemma:chopping-v}
Let $\widetilde{\mathcal M}_\ell^1 = \widetilde{\mathcal N}_\ell^1\setminus
\widetilde{\mathcal N}_{\ell-1}^1 \cup 
\{z \in \widetilde{\mathcal N}_\ell^1 \cap \widetilde{\mathcal N}_{\ell-1}^1\,|\, 
\omega_{\ell-1}(z) \subsetneq \omega_\ell(z)\}$. Then, we have  
$\operatorname{card} \widetilde{\mathcal M}_\ell^1 \leq C 
\operatorname{card} \widetilde{\mathcal N}_\ell^1\setminus\widetilde{\mathcal N}_{\ell-1}^1$
for a $C > 0$ depending only on the shape regularity of the triangulations. 
\end{enumerate}
\end{lemma} 
\begin{proof}
For statement (\ref{item:lemma:chopping-ii}), 
we only show the case $m  =1$ as the general case follows by induction. 
we note that $T \in \Tt_\ell \cap \TT$ implies that $T \not\in \TB_{2,\ell}$, where 
$\TB_{i,\ell}$, $ \in \{1,2,3\}$ are the three sets given in (\ref{eq:def-fcc}). 
If $T \in \TB_{3,\ell}$, then $T \in \TB_{1,\ell+1}$. 
If $T \in \TB_{1,\ell}$, then, $T \in \TB_{1,\ell+1}$. 
For statement (\ref{item:lemma:chopping-iii}), we have $T \in \Th_{\ell}\setminus \TT$. Then, 
$T \not \in \Tt_{\ell+1}$. 

For statement (\ref{item:lemma:chopping-iv}), let $z \in \widetilde{\mathcal N}_\ell^1$
and $T \in \widetilde\TT_\ell$ be an element such that $z$ is a node of $T$. 
We consider two cases.  First, if $T \in \TT \cap \widetilde \TT_\ell$, then, by statement 
 (\ref{item:lemma:chopping-ii}), we have $T \in \TT_{\ell+1}$ so that $z \in \widetilde {\mathcal N}_{\ell+1}^1$. 
Second, let $T \in \widetilde \TT_\ell\setminus \TT$. Then $T \in \widehat\TT_\ell$ and in fact 
in $\TB_{2,\ell}$. The node $z$ is the node of an element $T' \in \widehat \TT_{\ell+1}$. This 
element $T'$ is either in $\TT$, which implies $z \in \widetilde{\mathcal N}_{\ell+1}^1$, or 
$T' \in \TB_{2,\ell+1}$, which also implies $z \in \widetilde{\mathcal N}_{\ell+1}^1$. 

For statement (\ref{item:lemma:chopping-v}) one observes that 
$\operatorname{card} \{z \in \widetilde{\mathcal N}_\ell^1\cap \widetilde{\mathcal N}_{\ell-1}^1\,|\, 
\omega_{\ell-1}(z) \subsetneq \omega_\ell(z)\} \lesssim
 \operatorname{card}\{T \in \widetilde \TT_{\ell-1}\,|\, T \not\in \widetilde\TT_\ell\} 
\lesssim \operatorname{card} \widetilde{\mathcal N}_\ell^1\setminus \widetilde {\mathcal N}_{\ell-1}^1$. 
\end{proof}

The following lemma shows that the adapted Scott-Zhang operators for the meshes $\Tt_\ell$ and $\Th_\ell$ 
coincide on piecewise polynomials on the mesh $\TT$. 
\begin{lemma}
\label{lemma:scott-zhang-tilde-hat}
Let $\TT$ be generated by NVB from $\Th_0$. 
Let  $\widetilde{I}^{SZ}_\ell:L^2(\Omega) \rightarrow S^{p,1}(\widetilde \TT_\ell)$ 
and $\widehat I^{SZ}_\ell:L^2(\Omega)\rightarrow S^{p,1}(\Th_\ell)$ be 
the Scott-Zhang operators defined in Definition~\ref{def:adapted-SZ}. 
Then, there holds
$$
\widetilde I^{SZ}_\ell u = \widehat I^{SZ}_\ell u 
\qquad \forall u \in S^{p,1}(\TT). 
$$
\end{lemma}
\begin{proof}
\emph{1.~step:} Let $T \in \Th_\ell \cap \Tt_\ell$. We claim 
that $(\widetilde I^{SZ}_\ell u)|_T = (\widehat I^{SZ}_\ell u)|_T$. The nodes 
$z \in \overline{T}$ and the shape functions $\varphi_{z,\Th_\ell}$, $\varphi_{z,\Tt_\ell}$ 
for the meshes $\Th_\ell$ and $\fcc(\TT,\Th_\ell)$ coincide on $T$. 
For the averaging element $T_z$ associated with $z \in \overline{T}$, 
two cases can occur: 
\begin{enumerate}
\item 
The two averaging sets for the two operators coincide. This happens 
in the following three cases: 
a) 
if $z \in T$ (case~\ref{item:def-SZ-1a} of Def.~\ref{def:adapted-SZ}); 
b) if $z \in \partial T$ and 
(case~\ref{item:def-SZ-1bii} of Def.~\ref{def:adapted-SZ}) 
arose for $T$ in the loop; 
c) 
(case~\ref{item:def-SZ-1bii} of Def.~\ref{def:adapted-SZ}) 
arose for an element $T' \in \Th_\ell \cap \Tt_\ell$ with $z \in \overline{T'}$ that
appeared earlier in the loop than $T$. 
Since the averaging sets coincide, the value of the 
linear functionals are the same.  
\item 
Case~\ref{item:def-SZ-1bi} of Def.~\ref{def:adapted-SZ} arose. Then,
both averaging sets are contained in an element $\widetilde T \in \TT$.
Since $u|_{\widetilde T} \in P_p$, we obtain from 
(\ref{eq:averaging-gleich-punktauswertung}) that both linear functionals
equal $u(z)$. 
\end{enumerate}
Hence, in all cases the values of the linear functionals coincide so that 
indeed the Scott-Zhang operators on the element $T$ are equal. 

\emph{2.~step:} In the region not covered by elements in $\Th_\ell \cap \Tt_\ell$
we show $\widetilde I^{SZ}_\ell u = u$ and $\widehat I^{SZ}_\ell u = u$ for 
$u \in S^{p,1}(\TT)$. 
For $\widetilde I^{SZ}_\ell$ this is shown in step 3 and for $\widehat I^{SZ}_\ell$ in step 4. 
This completes the proof of the lemma. 

\emph{3.~step:} 
We start by noting that the definition of the finest common coarsening implies
\begin{equation}
\label{eq:foo}
\mbox{ for any $T' \in \Th_\ell \setminus \Tt_\ell$ there exists
$\widetilde T \in \TT$ with $T' \subset \widetilde T$}. 
\end{equation}
Consider now $T \in \Th_\ell \setminus \Tt_\ell$. By
(\ref{eq:foo}) there exists 
$\widetilde T \in \TT$ such that $T \subset \widetilde T$. For 
$u \in S^{p,1}(\TT)$ we have $u|_{\widetilde T} \in P_p(\widetilde T)$. 
Moreover, $(\widehat I^{SZ}_\ell u)|_T  = \sum_{z\in\mathcal{N}(T)} \varphi_{z,\Th_\ell} l_z(u)$
with the linear functional $l_z(u) = \int_{T_z} \varphi^\ast_{z,T} u$. 
For the interior nodes $z \in T$ we have $T_z = T$ and, since $u|_T \in P_p(T)$, 
$l_z(u) = u(z)$ by (\ref{eq:averaging-gleich-punktauswertung}). 
For $z \in \partial T$, the following cases may occur: 
\begin{enumerate}[(1)]
\item 
\label{item:step3-1}
If $T_z = T$, then again 
$l_z(u) = u(z)$ by (\ref{eq:averaging-gleich-punktauswertung}). 
\item 
\label{item:step3-2}
If $T_z $ is a neighboring element of $T$, then the following cases can occur:
  \begin{enumerate}[(a)]
    \item 
\label{item:step3-2a}
      $T_z \in \Th_\ell \cap \Tt_\ell$: Then, 
       $z \in \partial T$ and hence also in $\partial T_z$. The construction 
       of the averaging sets in Def.~\ref{def:adapted-SZ} is such that the 
       averaging set $T_z$ for node $z$ is chosen such that it is contained in 
       an element $T' \in \TT$ if possible. Since $T \subset \widetilde T \in \TT$ 
       is possible by (\ref{eq:foo}), we conclude that also 
       $T_z \subset T'' \in \TT$ for some $T'' \in \TT$. Hence, $u|_{T_z} \in P_p(T_z)$, and the value of 
       the linear functional is $u(z)$. 
     \item 
\label{item:step3-2b}
          $T_z \in \Th_\ell \setminus \Tt_\ell$. Then, 
           by (\ref{eq:foo}) we get $u|_{T_z} \in P_p(T_z)$ so that again 
           by (\ref{eq:averaging-gleich-punktauswertung}) $l_z(u) = u(z)$.
        \end{enumerate}
  
\end{enumerate}
In total, we have   
arrived at $(\widehat  I^{SZ}_\ell u)|_T = \sum_{z \in \mathcal{N}(T)} 
\varphi_{z,\Th_\ell} u(z) = u|_T$, since $u|_T \in P_p(T)$. 

\emph{4.~step:} 
Consider $T \in \Tt_\ell\setminus \Th_\ell$. Then $T \in \TT$. 
We have $(\widetilde I^{SZ}_\ell u)|_T  = \sum_{z \in\mathcal{N}(T)} \varphi_{z,\Tt_\ell} l_z(u)$
with the linear functional $l_z(u) = \int_{T_z} \varphi^\ast_{z,T} u$. 
For the interior nodes $z \in T$ we have $T_z = T$ and, since $u|_T \in P_p(T)$, 
the property (\ref{eq:averaging-gleich-punktauswertung}) gives 
$l_z(u) = u(z)$.  

For $z \in \partial T$, two cases may occur: If $T_z = T$, then again 
$l_z(u) = u(z)$ by (\ref{eq:averaging-gleich-punktauswertung}). 
If $T_z$ is a neighboring element of $T$, then either 
$T_z \in \Th_\ell \cap \Tt_\ell$, which means 
$l_z(u) = u(z)$ by the same reasoning as in
step~3, item~\ref{item:step3-2a}, 
or $T_z \in \Tt_\ell\setminus \Th_\ell \subset \TT$ so 
that $u|_{T_z} \in P_p(T_z)$ and thus 
by (\ref{eq:averaging-gleich-punktauswertung}) $l_z(u) = u(z)$.
In total, we have
arrived at $(\widetilde  I^{SZ}_\ell u)|_T = \sum_{z \in \mathcal{N}(T)} 
\varphi_{z, \Tt_\ell} u(z) = u|_T$, since $u|_T \in P_p(T)$. 
\end{proof}
\subsection{Proof of the norm equivalence of Theorem~\ref{thm:multilevel-fcc}}
With Lemma~\ref{lemma:scott-zhang-tilde-hat}, Corollary~\ref{cor:interpolation}, and Lemma~\ref{lemma:inverse-estimate},
we are able to prove 
the norm equivalence for the multilevel decomposition of Theorem~\ref{thm:multilevel-fcc}.

\begin{proof}[Proof of Theorem~\ref{thm:multilevel-fcc}]
We apply \cite[Thm.~{3.5.3}]{cohen2003numerical} for the spaces 
$X = \Big(S^{p,1}(\TT),\norm{\cdot}_{L^2(\Omega)}\Big)$, 
$Y = \Big(S^{p,1}(\TT),\norm{\cdot}_{B^{3/2}_{2,\infty}(\Omega)}\Big)$ noting that
we have $S^{p,1}(\Tt_\ell) \subset S^{p,1}(\TT)$. 

Then, \cite[Thm.~{3.5.3}]{cohen2003numerical} provides the equivalence of the second and third norm to the 
norm on the interpolation space $(X,Y)_{\theta,q}$, which by Corollary~\ref{cor:interpolation} is the 
$B^{3/2\theta}_{2,q}(\Omega)$-norm, provided a Jackson-type and a Bernstein-type estimate holds. \\

\emph{1.~step (Jackson-type inequality):} 
Using Lemma~\ref{lemma:scott-zhang-tilde-hat}, we compute 
for $u \in S^{p,1}(\TT)$ and arbitrary $w \in S^{p,1}(\Th_\ell)$
\begin{align*}
\inf_{v \in S^{p,1}(\Tt_\ell) } \|u - v\|_{L^2(\Omega)} & \leq 
\|u - \widetilde I^{SZ}_\ell u\|_{L^2(\Omega)}   = 
\|u - \widehat I^{SZ}_\ell u\|_{L^2(\Omega)}   = 
\|u -w - \widehat I^{SZ}_\ell (u-w)\|_{L^2(\Omega)}   \\
&\lesssim \|u - w\|_{L^2(\Omega)}. 
\end{align*}
Hence, standard approximation results on quasi-uniform meshes provide
\begin{align}
\nonumber 
 \inf_{v \in S^{p,1}(\Tt_\ell) } \|u - v\|_{L^2(\Omega)} \leq \|u - \widetilde I^{SZ}_\ell u\|_{L^2(\Omega)} &\lesssim
\inf_{w \in \in S^{p,1}(\Th_\ell)} \|u - w\|_{L^2(\Omega)}  \\
\label{eq:jackson}
&\lesssim \hh_\ell^{3/2}\|u\|_{B^{3/2}_{2,\infty}(\Omega)} \lesssim 
2^{-3\ell /2}\|u\|_{B^{3/2}_{2,\infty}(\Omega)}
\end{align}
since $\Th_\ell$ is a quasi-uniform mesh of mesh size $\hh_\ell = \hh_02^{-\ell}$.
We note that this estimate also implies the additional assumption \cite[Eqn.(3.5.29)]{cohen2003numerical}
on the projection operators $\widetilde I^{SZ}_\ell$.\\

\emph{2.~step (Bernstein-type inequality):} 
Using the projection property of the Scott-Zhang operators 
and Lemma~\ref{lemma:scott-zhang-tilde-hat}, we get for 
arbitrary $v \in S^{p,1}(\Tt_\ell)$
\begin{align}
\nonumber 
\|v\|_{B^{3/2}_{2,\infty}(\Omega)} & = 
\|\widetilde I^{SZ}_\ell v\|_{B^{3/2}_{2,\infty}(\Omega)}  = 
\|\widehat I^{SZ}_\ell v\|_{B^{3/2}_{2,\infty}(\Omega)}  
\stackrel{\text{Lemma~\ref{lemma:inverse-estimate}}}{\lesssim} \hh_\ell^{-3/2} 
\|\widehat I^{SZ}_\ell v\|_{L^{2}(\Omega)}  \\
\label{eq:bernstein}
 & = \hh_\ell^{-3/2}
\|\widetilde  I^{SZ}_\ell v\|_{L^{2}(\Omega)} 
 = \hh_\ell^{-3/2}
\|v\|_{L^{2}(\Omega)}. 
\end{align}
As the family of operators $\widetilde I^{SZ}_{\ell}:X\rightarrow S^{p,1}(\Tt_\ell)$ is also uniformly bounded in
the $L^2(\Omega)$-norm, all assumptions of \cite[Thm.~{3.5.3}]{cohen2003numerical} are valid and 
consequently the norm equivalences are proven.
\end{proof}

  \subsection{Boundary conditions}
  The previous results do not consider (homogeneous) Dirichlet boundary conditions. For 
  the application we have in mind (cf.\ \eqref{eq:modelproblem}), an interpolation result similar to 
  Corollary~\ref{cor:interpolation} for the spaces $L^2(\Omega)$, $H^1_0(\Omega)$ and $\widetilde{H}^s(\Omega)$ for
  $s \in (0,1)$ is of interest. Interpolation results for theses spaces are already available in the literature,
  see, e.g., \cite{AFFKP15}, where the proof uses stability properties of the Scott-Zhang projection and 
  the abstract result from \cite{arioli-loghin09}, similarly to Corollary~\ref{cor:interpolation}. 
  For sake of completeness, we state the result in the following corollary.

\begin{corollary}\label{cor:interpolation-0} 
Let $s\in (0,1)$. Then, there holds
\begin{align*}
\left( (S_0^{p,1}(\TT), \|\cdot\|_{L^2(\Omega)}),  
(S_0^{p,1}(\TT), \|\cdot\|_{H^{1}(\Omega)})
\right)_{s,2} = (S_0^{p,1}(\TT), \|\cdot\|_{\widetilde{H}^s(\Omega)})
\end{align*}
with equivalent norms. 
\end{corollary} 

As done, for example, in \cite{AFFKP15}, the Scott-Zhang operators 
$\widetilde I^{SZ}_\ell$  and $\widehat I^{SZ}_\ell$ can be modified by simply dropping 
the contributions from the shape functions associated with nodes on $\partial\Omega$ and thus map
into the spaces 
$\widetilde S^{p,1}_0(\widetilde\TT_\ell)$ and  $\widetilde S^{p,1}_0(\widehat\TT_\ell)$, respectively. 
We denote these operators by $\widetilde I^{SZ}_{0,\ell}$ and $\widehat I^{SZ}_{0,\ell}$, and
they are still stable in $L^2(\Omega)$ and $H^1_0(\Omega)$. 
Therefore, Theorem~\ref{thm:multilevel-fcc} also provides a lower bound for the multilevel decomposition 
based on the Scott-Zhang operator in the $\widetilde{H}^s(\Omega)$-norm.
  
  \begin{corollary} \label{cor:lowerbound}
  Let $\TT$ be a mesh obtained by NVB-refinement of a triangulation $\Th_0$.
  Let $\Th_\ell$ be the sequence of uniformly refined meshes starting from $\Th_0$
  with mesh size $\hh_\ell = \hh_0 2^{-\ell}$. 
  Set $\Tt_\ell:= \fcc(\TT,\Th_\ell)$. 
  Let $\widetilde{I}^{SZ}_{0,\ell}:\widetilde H^s(\Omega) \rightarrow S^{p,1}_0(\Tt_\ell)$ 
  be the Scott-Zhang operator defined as above. 
Then, we have   	
\begin{align} \label{ Pieswise linear-Projection}
\sum_{\ell=0}^{\infty} \widehat{h}_\ell ^{-2s}\left\| u- \widetilde{I}^{SZ}_{0,\ell} u\right\| ^2 _{L^2 (\Omega)} 
\le C \left\|u \right\| ^2 _ {\widetilde H^{s}(\Omega)} \quad \quad \forall u \in S _0^{p,1} (\mathcal{T})   \quad   0<s<1.
\end{align}
  \end{corollary}
  \begin{proof}
  We note that Jackson-type and Bernstein-type estimates (\ref{eq:jackson}) and (\ref{eq:bernstein}) in
  the proof of Theorem~\ref{thm:multilevel-fcc} also hold for the variant of the Scott-Zhang projection that preserves
  homogeneous boundary conditions if we replace $\hh^{3/2}_\ell \|u\|_{B^{3/2}_{2,\infty}(\Omega)}$ with 
$\hh_\ell \|u\|_{H^1_0(\Omega)}$ in (\ref{eq:jackson}) and if we replace in (\ref{eq:bernstein}) 
the norms $\|\cdot\|_{B^{3/2}_{2,\infty}(\Omega)}$ with $\|\cdot\|_{H^1(\Omega)}$ and correspondingly 
$\hh^{-3/2}$ with $\hh^{-1}$. Therefore, the norm equivalences of Theorem~\ref{thm:multilevel-fcc} are 
still valid if one replace $B^{3\theta/2}_{2,\infty}(\Omega)$ with $H^{\theta}_0(\Omega)$, 
$\widetilde I^{SZ}_\ell$ with 
$\widetilde I^{SZ}_{0,\ell}$, and $2^{3\theta\ell/2}$ with $2^{\theta \ell}$.  
\end{proof}

\section{Optimal additive Schwarz  preconditioning for the  fractional Laplacian on locally refined meshes}
\label{sec:ASproof}
In this section, we prove 
the optimal bounds on the eigenvalues of the preconditioned matrices $\mathbf{P}_{AS}^{L}$
of Theorem~\ref{thm:ASfractional} and $\widetilde{\mathbf{P}}_{AS}^{L}$
of Theorem~\ref{thm:ASfractional2}. The key steps are done in Proposition~\ref{prop:spectralEquiv}
or Proposition~\ref{prop:Equiv}, which  
state a
spectral equivalence of the corresponding additive Schwarz operator and the identity in the energy scalar product.

\subsection{Abstract analysis of the additive Schwarz method}
\label{sub sec:Abstract analysis of the additive Schwarz}
\subsubsection{The mesh hierarchy $\widetilde \TT_\ell = \fcc(\TT_L,\widehat\TT_\ell)$}

The additive Schwarz method is based on a local subspace decomposition.
For the mesh hierarchy $\widetilde \TT_\ell = \fcc(\TT_L,\widehat\TT_\ell)$
we recall that $\widetilde{V}_\ell \in \{S^{0,0}(\widetilde \TT_\ell),S^{1,1}_0(\widetilde \TT_\ell)\}$
is either the space of piecewise constants or piecewise linears on the mesh $\widetilde \TT_\ell$.
We follow the abstract setting 
of \cite{TosWid05} and decompose
$\widetilde V_L = \sum_{\ell=0}^L \widetilde{\mathcal{V}}_\ell$
with
\begin{align}\label{eq:defdecomp}
\widetilde{\mathcal{V}}_\ell := \operatorname*{span}\left\{ \widetilde\varphi^{\ell}_z \; : \; z \in \widetilde{\mathcal{M}}_\ell\right\},
\end{align}
where $\widetilde\varphi ^{\ell}_z$
denotes the basis function associated with the node $z \in \widetilde{\mathcal N}_\ell$. We recall that these 
functions are either characteristic functions of elements (for the piecewise constant case) or nodal 
hat functions (for the case of piecewise linears).
We note that 
 $\widetilde{\mathcal{V}}_\ell \subset \widehat{V}_\ell$ and since $\widetilde{\mathcal{M}}_\ell$ only contains  new nodes 
and direct neighbors this space effectively is a discrete space on a uniform submesh (cf. Lemma~\ref{lem:Ml}).

On the subspaces $\widetilde{\mathcal{V}}_\ell$ we introduce the symmetric, positive definite bilinear form  
$\widetilde a_\ell(\cdot,\cdot):\widetilde{\mathcal{V}}_\ell \times \widetilde{\mathcal{V}}_\ell$ (also known as local solvers) with 
\begin{align*}
 \widetilde a_\ell(u_\ell,u_\ell) := \sum_{z \in \widetilde{\mathcal{M}}_\ell} 
 \norm{\widehat h_\ell^{-s}u_\ell(z)\widetilde\varphi^{\ell}_z}_{L^2(\Omega)}^2 \simeq 
 \sum_{z \in \widetilde{\mathcal{M}}_\ell}  \widehat h_\ell^{d-2s}\abs{u_\ell(z)}^2.
\end{align*}

The following proposition, c.f., e.g., \cite{zhang92,MatNep85}, 
gives bounds on the minimal and maximal 
eigenvalues of the preconditioned matrix $\widetilde{\mathbf{P}}^L_{AS}$ based on the abstract additive Schwarz theory.

\begin{proposition}\label{prop:Equiv}
 \begin{enumerate}
  \item[(i)] Assume that every $u \in \widetilde V_L$ admits a decomposition $u = \sum_{\ell=0}^L u_\ell$ with 
  $u_\ell \in \widetilde{\mathcal{V}}_\ell$ satisfying
  $ \sum_{\ell=0}^L \widetilde a_\ell(u_\ell,u_\ell) \leq C_{0}\; a(u,u) 
  $
  with a constant $C_0 > 0$.
  Then, we have $\lambda_{\rm min}(\widetilde{\mathbf{P}}^L_{AS}) \geq C_{0}^{-1} $.
  
  \item[(ii)] Assume that there exists a constant $C_{1}>0$ such that for every decomposition
  $u = \sum_{\ell=0}^L u_\ell$ with $u_\ell \in \widetilde{\mathcal{V}}_\ell$, we have 
  $a(u,u)  \leq C_{1}  \sum_{\ell=0}^L \widetilde a_\ell(u_\ell,u_\ell).
  $
  Then, $\lambda_{\rm max}(\widetilde{\mathbf{P}}^L_{AS}) \leq C_{1} $.
 \end{enumerate}
\end{proposition}

The first part of Proposition~\ref{prop:Equiv} is sometimes called Lions' Lemma and follows from the existence of a stable 
decomposition proven in Lemma~\ref{lem:stabledec} below. 
The assumption of the second statement follows directly from our strengthened Cauchy-Schwarz inequality 
(Lemma~\ref{lem:strengthenedCS}) and local stability (Lemma~\ref{lem:localstab}). 
\bigskip 

It remains to show the assumptions of Proposition~\ref{prop:Equiv}.
A key ingredient in the proof is an inverse estimate 
        for the fractional Laplacian provided in Subsection~\ref{sec:invest}.
        

\subsubsection{The mesh hierarchy $\TT_\ell$ provided by an adaptive algorithm}

For the case of a mesh hierarchy  $\mathcal{T}_\ell$ generated by an adaptive algorithm similar definitions 
can be made and analyzed. However, here, we follow the notation of \cite{feischl2017optimal}, where the 
additive Schwarz operator consisting of a sum of
projections onto one dimensional spaces is analyzed. With the spaces 
$V_z ^{\ell} := \operatorname*{span}\{\varphi_z^\ell\}$ one may define local projections
$\mathcal{P}_z ^{\ell} : \widetilde H^s (\Omega) \rightarrow V_z ^{\ell}$ in the energy scalar product as 
\begin{align*}
a(\mathcal{P}_z^{\ell} u,v_z^\ell)
= a(u,v_z^{\ell}) \qquad  \text{for all} \; v_z^{\ell} \in V_z^{\ell}
\end{align*}
and define the additive Schwarz operator as 
\begin{align*}
\mathcal{P}_{AS}^{L} := \sum_{\ell=0}^{L} \sum_{z \in \mathcal{M}_\ell} \mathcal{P} _z ^{\ell}.
\end{align*}

Moreover, for $u$, $v \in V_L$ and their expansions 
$u=\sum_{j=1}^{N_L} \mathbf{x}_j \varphi_{z_j}^{L}$, $v=\sum_{j=1}^{N_L} \mathbf{y}_j \varphi_{z_j}^{L}$, we have 
\begin{align}\label{eq:matrixOperator}
a(\mathcal{P}_{AS}^{L} u,v) = 
\left\langle \mathbf{P}_{AS}^{L}\mathbf{x},\mathbf{y}\right\rangle_ {\mathbf{A} ^L},
\end{align}
where $\skp{\cdot,\cdot}_{\mathbf{A}^L} := \skp{\mathbf{A}^L\cdot,\cdot}_2$.
Therefore, the multilevel diagonal scaling is a multilevel additive Schwarz method.
Due to this observation we may analyze the additive Schwarz operator instead of the preconditioned matrix. 

  \begin{proposition}\label{prop:spectralEquiv}
  	The operators $ \mathcal{ P}_{AS} ^{L}$ is linear, bounded and symmetric in the energy scalar product.
  	 Moreover, for $u \in V_L$, we have the spectral equivalence
  	\begin{align} \label{eq:spectralEquiv}
  	c\norm{u}_{\widetilde H^s(\Omega)}^2 \le a(\mathcal{ P}_{AS}^{L} u,u) \le C \norm{u}_{\widetilde H^s(\Omega)}^2,
  	\end{align}
  	where the constants $c$, $C >0$ only depend on $\Omega$, $d$, $s$, and $\mathcal{T} _0$. 
  \end{proposition}
  
  As in \cite{feischl2017optimal}, Proposition~\ref{prop:spectralEquiv} directly implies 
  Theorem~\ref{thm:ASfractional}.

\begin{proof}[Proof of Theorem~\ref{thm:ASfractional}]
Combining the bounds of Proposition~\ref{prop:spectralEquiv} with \eqref{eq:matrixOperator} gives 
\begin{align*}
c\norm{\mathbf{x}}_ {\mathbf{A} ^L}^2 \leq 
\left\langle  \mathbf{P}_{AS}^{L}\mathbf{x},\mathbf{x}\right\rangle_ {\mathbf{A} ^L}
\leq C\norm{\mathbf{x}}_ {\mathbf{A} ^L}^2
\end{align*}
for all $\mathbf{x} \in \mathbb{R}^{N_L}$, and therefore the bounds for the minimal and maximal eigenvalues.
Clearly, the same holds for $\widetilde{\mathbf{P}}_{AS}^{L}$.
\end{proof}


\subsection{Inverse estimates for the fractional Laplacian} 
\label{sec:invest}

	In order to prove a strengthened Cauchy Schwarz inequality, an inverse inequality for the operator 
	$(-\Delta)^s$ of the form
	\begin{align}\label{eq:investS11}
	\norm{h^s(-\Delta)^s v}_{L^2(\Omega)} \lesssim \norm{v}_{\widetilde{H}^s(\Omega)}
	\end{align}
	is used, where $h \in L^\infty(\Omega)$ denotes the piecewise constant mesh width function 
	on a regular triangulation $\TT$ generated by NVB refinement of a given regular triangulation $\TT_0$.
	For the piecewise linear case $v \in S^{1,1}_0(\TT)$, this inverse estimate is proven in 
	\cite[Thm.~2.7]{faustmann2019quasi}.
	We stress that \eqref{eq:investS11} only holds for $s<3/4$, since in the converse case the 
	left-hand side is not well defined for $v \in S^{1,1}_0(\TT)$. To obtain an estimate for $s \in [3/4,1)$ 
one has to introduce a weight function 
	$w(x):= \operatorname*{inf}_{T \in \TT} \operatorname*{dist}(x,\partial T)$. Then, 
	\cite[Thm.~2.7]{faustmann2019quasi} provides the inverse estimate 
        \begin{align}\label{eq:investS12}
	  \norm{h^{1/2}w ^{s-1/2} (-\Delta)^s v}_{L^2(\Omega)} \lesssim 
	  \norm{v}_{\widetilde{H}^s(\Omega)}.
	\end{align}	
	For the case of piecewise constants, similar inverse estimates are stated in the lemma below. Here,
	we additionally stress that for $v \in S^{0,0}(\TT)$ and $x \in T \in \TT$ the estimate
\begin{align}\label{eq:L2condition}
\abs{(-\Delta)^s v(x)} &= 
\left| 
\int_{\R^d\backslash B_{\operatorname*{dist}(x,\partial T)}(x)}\frac{v(x)-v(y)}{\abs{x-y}^{d+2s}}dy
\right|
\lesssim \norm{v}_{L^{\infty}(\Omega)}
\int_{B_{\operatorname*{dist}(x,\partial T)}(x)^c} \frac{1}{\abs{x-y}^{d+2s}}dy\nonumber \\
&= \norm{v}_{L^{\infty}(\Omega)}
\int_{\nu \in \partial B_1(0)}
\int_{r = \operatorname*{dist}(x,\partial T) }^{\operatorname{diam}\Omega}r^{-2s-1} 
drd \nu 
\lesssim \norm{v}_{L^{\infty}(\Omega)}\operatorname*{dist}(x,\partial T)^{-2s}
\end{align}
gives 
	\begin{align*}
	w^\beta (-\Delta)^s v \in L^2(\Omega) \qquad \text{if} \; \; \beta > 2s-1/2.
	\end{align*}
For $s<1/4$, we may choose $\beta =0$ and for $1/4\leq s <1/2$, we may choose, e.g., $\beta = s$
or $\beta = 3/2s-1/4$ (to additionally ensure $\beta < s$)
to fulfill this requirement.

  \begin{lemma} \label{lem:invest}
  Let $\TT$ be a regular and $\gamma$-shape regular mesh generated by NVB refinement of a mesh $\TT_0$.
  	Let $v \in S^{0,0} (\mathcal{T})$,
	$h$ be the piecewise constant mesh width function 
	of the triangulation $\TT$, and set 
	$w(x):= \operatorname*{inf}_{T \in \TT} \operatorname*{dist}(x,\partial T)$.	
	Let $\beta > 2s-1/2$.
	Then, the inverse estimates
  	\begin{align}
  	\label{inverse estimate 1}
  	\| h ^s  (- \Delta)^s v\|_ {L^2 (\Omega)} & \leq C\| v\|_ {\widetilde{H} ^s  (\Omega)} \qquad 0 < s <  1/4,\\
  	\label{inverse estimate 2}
   	\| h ^{s-\beta} w ^{\beta}  (- \Delta)^s v\|_ {L^2 (\Omega)} &\leq C \| v\|_ {\widetilde{H}^s(\Omega)} \qquad  1/4 \leq s< 1/2
   	\end{align}	
		hold, where the constant $C>0$ depends only on $\Omega$, $d$, $s$, and 
		the $\gamma$-shape regularity of $\TT$.
  \end{lemma}
 \begin{proof}
If we set $\beta = 0$ for $s<1/4$, 
we can prove both statements of the lemma at once by estimating the $L^2$-norms with the weight
 $h  ^{s- {\beta}} w ^{\beta}$.
 We follow the lines of \cite[Thm.~{2.7}]{faustmann2019quasi} starting with a splitting into a near-field and a far-field 
 part.

 For each $T \in \mathcal{T}$, we choose a cut-off function $\chi _T \in C ^ \infty _0 (\mathbb{R}^d)$
 with the properties: \linebreak 1) $\support  \chi _T \subset \omega  (T)$; 
 2) $\chi _T \equiv 1$ on a set $B$ satisfying $T \subset B \subset \omega  (T)$ and 
 $\operatorname*{dist}(B , \partial \omega(T)) \sim h_T$; 
 3) $\| \chi _T \| _ {W ^ {1, \infty}(\omega (T))} \lesssim h_T^{-1}$. 
 
 Moreover, for each $T \in \mathcal{T}$, we denote the average of $v$ on the patch $\omega(T)$
 by $c_T \in \R$. Since $c_T$ is a constant, we have  $(-\Delta)^s c_T \equiv 0$.
 Therefore, we can decompose $v$
  into the near-field 
 $ v^T _ {\text{near}} :=\chi _T(v-c_T)$ and the far-field $v^T _ {\text{far}} :=(1-\chi _T)( v-c_T)$,
 and obtain $(-\Delta)^s v =  (-\Delta)^s v^T_{\text{near}}+ (-\Delta)^sv^T _ {\text{far}}$.\\

 The estimates of the near-field and the far-field are rather similar to the case of piecewise linears 
 from \cite[Lem.~4.2, Lem.~4.4]{faustmann2019quasi}. Therefore, we quote the identical parts of the proof and outline the necessary 
 modifications for the piecewise constant case.
 
 We start with the near-field, where compared to the result for the case of piecewise linears, we do not 
 need to distinguish cases for $s$.
  The definition of the 
fractional Laplacian leads to
\begin{align}\label{eq:investtmp1}
\norm{w^\beta(-\Delta)^sv^T_{\text{near}}}_{L^2(T)}^2 &= 
\int_T w(x)^{2\beta}\left(\text{P.V.}\int_{\R^d}\frac{(v(x)-c_T)\chi_T(x)-(v(y)-c_T)\chi_T(y)}{\abs{x-y}^{d+2s}} dy\right)^2 dx
\nonumber \\& \lesssim 
\int_T w(x)^{2\beta}(v(x)-c_T)^2\left(\text{P.V.}\int_{\R^d}\frac{\chi_T(x)-\chi_T(y)}{\abs{x-y}^{d+2s}} dy\right)^2 dx
\nonumber \\& \quad
+\int_T w(x)^{2\beta}\left(\text{P.V.}\int_{\R^d}\chi_T(y)\frac{v(x)-v(y)}{\abs{x-y}^{d+2s}} dy\right)^2 dx. 
\end{align}
The first term on the right-hand side can be estimated using the Lipschitz continuity of $\chi_T$ and a 
Poincar\'e inequality on the patch $\omega(T)$ in the same way as in the proof of 
\cite[Lem.~4.2]{faustmann2019quasi} by 
\begin{align*}
\int_T w(x)^{2\beta}(v(x)-c_T)^2\left(\text{P.V.}\int_{\R^d}\frac{\chi_T(x)-\chi_T(y)}{\abs{x-y}^{d+2s}} dy\right)^2 dx
\lesssim h_T^{2\beta-2s}\norm{v}_{H^s(\omega(T))}^2.
\end{align*}
For the second term in \eqref{eq:investtmp1}, we stress that the integrand vanishes for $y\in T$ since $v$
is piecewise constant, and employ the same estimate as for \eqref{eq:L2condition}
to obtain 
\begin{align*}
\int_T w(x)^{2\beta}\left(\text{P.V.}\int_{\R^d}\chi_T(y)\frac{v(x)-v(y)}{\abs{x-y}^{d+2s}} dy\right)^2 dx
\lesssim \norm{v-c_T}_{L^\infty(\omega(T))}^2\int_T w(x)^{2\beta-4s} dx.
\end{align*}
Here, we added and subtracted the constant $c_T$ in the integrand and used the support properties of $\chi_T$ to 
obtain the $L^\infty$-norm on the patch.
As by choice of $\beta$, we always have $2\beta-4s>-1$, the last integral exists, and we can further estimate 
using a classical inverse estimate and a Poincar\'e inequality
\begin{align*}
\norm{v-c_T}_{L^\infty(\omega(T))}^2\int_T w(x)^{2\beta-4s} dx &\lesssim 
h_T^{2\beta-4s+d}\norm{v-c_T}_{L^\infty(\omega(T))}^2 \\ &\lesssim 
h_T^{2\beta-4s}\norm{v-c_T}_{L^2(\omega(T))}^2 \lesssim 
h_T^{2\beta-2s}\norm{v}_{H^s(\omega(T))}^2. 
\end{align*}
Inserting everything into \eqref{eq:investtmp1}, multiplying with $h_T^{2s-2\beta}$ and summing over all elements
$T\in \TT$ gives the desired estimate for the near-field.
 
 The far-field can be estimated using the Caffarelli-Silvestre extension, cf.~\cite{CafSil07} combined 
 with a Caccioppoli-type inverse estimate for the solution of the extension problem with boundary data 
 $(1-\chi_T)(v-c_T)$ as in \cite{faustmann2019quasi}. In fact, we observe that 
 \cite[Lem.~{4.4}]{faustmann2019quasi} holds for arbitrary  $v \in \widetilde H^s(\Omega)$ and weight functions
 $w$ with non-negative exponent. This directly gives 
  \begin{align*}
 \sum_{T \in \mathcal{T}} \|h  ^{s- {\beta}} w ^{\beta} (- \Delta)^sv^T _ {\text{far}} \|^2 _ {L^{2}(T)} 
 \lesssim \| v  \|^2 _ {\widetilde{H}^{s}(\Omega)},
 \end{align*}
 and combining the estimates for near- and far-field proves the lemma.
 \end{proof}

\subsection{Proof of the assumptions of Proposition~\ref{prop:Equiv}}
In order to apply  Proposition~\ref{prop:Equiv}, we show 
the existence of a stable decomposition 
(Lemma~\ref{lem:stabledec}) and 
a strengthened Cauchy-Schwarz inequality (Lemma~\ref{lem:strengthenedCS}). \bigskip

The following result relates the Scott-Zhang operators on two consecutive levels 
and is a key ingredient for the proof of Lemma~\ref{lem:stabledec}.

\begin{lemma}
\label{lemma:Isz-difference}
Let $p=1$ and let $\widetilde{\mathcal N}_\ell^1$, $\widetilde{\mathcal M}_\ell^1$ be defined 
in Section~\ref{sec:LMD-fcc}. 
The Scott-Zhang operators $\widetilde I^{SZ}_\ell:L^2(\Omega) \rightarrow S^{1,1}(\widetilde\TT_\ell)$
can be constructed such that, additionally, they satisfy, for all $\ell \in \BbbN$ and all $u \in L^2(\Omega)$ 
\begin{equation}
\label{eq:SZzero}
(\widetilde I^{SZ}_\ell - \widetilde I^{SZ}_{\ell-1})u(z) = 0 \qquad \forall
z \in \widetilde{\mathcal N}_{\ell}^1\setminus \widetilde {\mathcal M}_\ell^1.
\end{equation}
Also the Scott-Zhang operators $\widetilde I^{SZ}_{0,\ell}:L^2(\Omega) \rightarrow S^{1,1}_0(\widetilde \TT_\ell)$ 
can be constructed such that 
(\ref{eq:SZzero}) holds with 
$\widetilde I^{SZ}_{\ell}$ and 
$\widetilde I^{SZ}_{\ell-1}$ replaced with 
$\widetilde I^{SZ}_{0,\ell}$ and 
$\widetilde I^{SZ}_{0,\ell-1}$, respectively. 
\end{lemma}
\begin{proof}
We only consider the case of the operators $\widetilde I^{SZ}_\ell$. We also recall that for the present
case $p=1$ the nodes coincide with the nodes of the triangulations. 
\newline
\emph{Step 1:} $z \in \widetilde {\mathcal N}_{\ell}^1 \setminus \widetilde{\mathcal M}_\ell^1$ implies
$z \in \widetilde {\mathcal N}_\ell^1 \cap \widetilde{\mathcal N}_{\ell-1}^1$. To see $z \in \widetilde{\mathcal N}_{\ell-1}^1$, we note 
$\widetilde{\mathcal N}_{\ell-1}^1 \subset \widetilde {\mathcal N}_\ell^1$ by 
Lemma~\ref{lemma:chopping}  and therefore that 
$z \in \widetilde {\mathcal N}_\ell^1 \setminus \widetilde{\mathcal M}_\ell^1 \subset 
\widetilde{\mathcal N}_\ell^1 \setminus (\widetilde{\mathcal N}_\ell^1\setminus \widetilde{\mathcal N}_{\ell-1}^1) = 
\widetilde{\mathcal N}_{\ell-1}^1$. 
\newline 
\emph{Step 2:} $z \in \widetilde {\mathcal N}_{\ell}^1 \setminus \widetilde{\mathcal M}_\ell^1$ implies
that all elements of the patches $\omega_\ell(z)$ and $\omega_{\ell-1}(z)$ are in $\TT$. To see this, 
we note $z \in \widetilde {\mathcal N}_\ell^1 \setminus\widetilde{\mathcal M}_\ell^1 \subset 
\widetilde{\mathcal N}_\ell^1\setminus \{z \in \widetilde{\mathcal N}_\ell^1 \cap \widetilde{\mathcal N}_{\ell-1}^1\,|\, 
\omega_\ell(z) \subsetneq \omega_{\ell-1}(z)\}$. If $\omega_{\ell-1}(z) = \omega_\ell(z)$, then all elements of 
$\omega_{\ell-1}(z) = \omega_\ell(z)$ must be elements of $\TT$. 
\newline
\emph{3.~step:} The basic idea for the choice of averaging sets $T_z$ in the construction of 
$\widetilde I^{SZ}_{\ell-1}$ and $\widetilde I^{SZ}_{\ell}$ in Def.~\ref{def:adapted-SZ} is to select 
an element of $\TT$ whenever possible. 

Our modified construction of the operators $\widetilde I^{SZ}_\ell$ is by induction on $\ell$ and 
carefully exploits the freedom left in the choice of the averaging sets $T_z$ in Def.~\ref{def:adapted-SZ}. 
We start with an $\widetilde I^{SZ}_0$ as constructed in Def.~\ref{def:adapted-SZ}. Suppose the averaging sets 
$T_z$ for $\widetilde \TT_{\ell-1}$ have been fixed. Effectively, Def.~\ref{def:adapted-SZ} performs a loop over 
all nodes of $\widetilde\TT_\ell$. When assigning an averaging set $T_z$ to a node 
$z \in \widetilde{\mathcal N}_\ell^1\setminus \widetilde{\mathcal M}_\ell^1$, we select as $T_z$ 
the element that has already been selected on the preceding level $\ell-1$. This is possible since 
$z \in \widetilde{\mathcal N}_\ell^1\setminus \widetilde{\mathcal M}_\ell^1$ implies $z \in \widetilde{\mathcal N}_{\ell-1}^1$
by Step~1 and by Step~2 we know that all elements of both $\widetilde \TT_{\ell-1}$ and $\widetilde \TT_{\ell}$ 
having $z$ as a vertex are elements of $\TT$. 
\end{proof}

The following lemma provides the existence of a stable decomposition for the mesh hierarchy generated 
by the finest common coarsening.
Rather than analyzing the $L^2$-orthogonal projection onto a space of piecewise polynomials on a uniform mesh, as in
\cite{feischl2017optimal}, we use the result of Corollary~\ref{cor:lowerbound}.

\begin{lemma}\label{lem:stabledec}
(Stable decomposition for the mesh hierarchy $(\widetilde\TT_\ell)_\ell$).
For every $u \in \widetilde{V}_{L}$  there is a decomposition
$
u = \sum_{\ell=0}^L  u_\ell 
$
with $u_\ell \in \widetilde{\mathcal{V}}_\ell$ satisfying the stability estimate 
\begin{align*}
\sum_{\ell=0}^{L}\widetilde a_\ell(u_\ell,u_\ell) = 
\sum_{\ell=0}^{L}\sum_{z \in \widetilde{\mathcal{M}}_\ell}\norm{\widehat{h}_\ell^{-s}u_\ell(z)\widetilde{\varphi}_z^\ell}_{L^2(\Omega)}^2 
\le C_{\rm stab}^2 \norm{u}_{\widetilde{H}^s(\Omega)}^2.
\end{align*} 
with a constant $C_{\rm stab}>0$ depending only on $\Omega,d,s$, and the initial triangulation $\TT_0$.
\end{lemma}
\begin{proof}
We only show the case of piecewise linears, the piecewise constant case is even simpler as the basis functions 
are $L^2$-orthogonal.

Let $\widetilde I^{SZ}_{0,\ell}:\widetilde H^s(\Omega)\rightarrow S_0^{1,1}(\widetilde \TT_\ell)$ be the
adapted Scott-Zhang projection 
from Definition~\ref{def:adapted-SZ} in the form given by Lemma~\ref{lemma:Isz-difference}. 
Set $\widetilde I^{SZ}_{0,-1}=0$. Then, we define 
\begin{align*}
u_\ell := \sum_{z \in \widetilde{\mathcal{M}}_\ell}(\widetilde I^{SZ}_{0,\ell} - \widetilde I^{SZ}_{0,\ell-1})u(z) \widetilde \varphi_z^{\ell}.
\end{align*}

	Since $(\widetilde I^{SZ}_{0,\ell} - \widetilde I^{SZ}_{0,\ell-1})u \in \widetilde{\mathcal V}_\ell$, we may decompose using a telescoping series and \eqref{eq:SZzero}
	\begin{align}\label{eq:stabledec}
	u =
 \widetilde I^{SZ}_{0,L} u =
 \sum_{\ell=0}^{L}(\widetilde I^{SZ}_{0,\ell} - \widetilde I^{SZ}_{0,\ell-1})u
	= 
	\sum_{\ell=0}^{L}\sum_{z \in \widetilde{\mathcal{M}}_\ell}(\widetilde I^{SZ}_{0,\ell} - \widetilde I^{SZ}_{0,\ell-1})u(z) \widetilde \varphi_z^{\ell}
	=\sum_{\ell=0}^Lu_\ell.
	\end{align}
	The standard scaling of the hat functions in $L^2$ provides
	\begin{align} \label{eq:lowerboundscaling}
	\norm{\widetilde \varphi _z ^{\ell}}_{L^2(\Omega)}^2 \simeq h_\ell (z)^{d},
	\end{align}
	with $h_\ell (z)$ denoting the maximal mesh width on the patch corresponding to the node $z$.
	In the following, we prove the stability of the decomposition \eqref{eq:stabledec}.
	With \eqref{eq:SZzero}, \eqref{eq:lowerboundscaling} and an inverse estimate --
	cf.~\cite[Proposition {3.10}]{dahmen2004inverse}, which provides an estimate for the 
	nodal value of a piecewise linear function on the mesh $\widetilde{\TT}_\ell$ 
	by its $L^2$-norm on the patch of the node -- we estimate
	\begin{align}\label{Lower bound 13}
	\nonumber \sum_{\ell=0}^{L}\sum_{z \in \widetilde{\mathcal{M}}_\ell}\norm{\widehat h_\ell^{-s}(\widetilde I^{SZ}_{0,\ell} - \widetilde I^{SZ}_{0,\ell-1})u(z) \widetilde \varphi _z^{\ell}}_{L^2(\Omega)}^2 
	&\lesssim \sum_{\ell=0}^{L}\widehat h_\ell^{-2s}\sum_{z \in \widetilde{\mathcal{M}}_\ell} h_\ell(z)^{d} | (\widetilde I^{SZ}_{0,\ell} - \widetilde I^{SZ}_{0,\ell-1})u(z)| ^2\\
	&\lesssim \sum_{\ell=0}^{L}\widehat h_\ell^{-2s}\sum_{z \in \widetilde{\mathcal{N}}_\ell}\norm{(\widetilde I^{SZ}_{0,\ell} - \widetilde I^{SZ}_{0,\ell-1})u}^2_ {L^2 (\omega_\ell(z))}\nonumber \\
        & \lesssim \sum_{\ell=0}^{L}\widehat{h}^ {-2s} _{\ell}\sum_{T\in \widetilde{\mathcal{T}}_{\ell}} \norm{(\widetilde I^{SZ}_{0,\ell} -\widetilde I^{SZ}_{0,\ell-1})u}^2_ {L^2 (T)}.  
	\end{align}
	Finally, we can use 
	Corollary~\ref{cor:lowerbound} to  obtain 
	\begin{align}\label{Lower bound 15}
	\sum_{\ell=0}^{L}\widetilde a_\ell(u_\ell,u_\ell)
	&\lesssim  \sum_{\ell=0} ^{L}\widehat{h}^{-2s}_{\ell} \norm{(\widetilde I^{SZ}_{0,\ell} -\widetilde I^{SZ}_{0,\ell-1})u}^2_ {L^2 (\Omega)}
	\lesssim \norm{u}_{\widetilde{H}^s(\Omega)}^2,
	\end{align} 
	which proves the existence of a stable decomposition.
\end{proof}

The following lemma shows that the submesh consisting of the elements corresponding 
to the points in $\widetilde{\mathcal{M}}_{\ell}$
is indeed quasi-uniform.

\begin{lemma} \label{lem:Ml}
Let $\widetilde{\mathcal{M}}_{\ell}$ be defined 
in Section~\ref{sec:LMD-fcc} and let $z \in \widetilde{\mathcal{M}}_{\ell}$, then it holds 
$h _{ \ell}(z) \simeq \widehat{h}_{\ell}$, where $h _{ \ell}(z) $ denotes the 
maximal mesh width on the patch $\omega_\ell(z)$. In particular, we have 
$\widetilde{\mathcal{V}_\ell} \subset \widehat{V}_\ell$, meaning 
$\widetilde{\mathcal{V}_\ell} \subset \widehat{V}_\ell^0$ if $\widetilde{\mathcal{M}}_{\ell} = 
\widetilde{\mathcal{M}}_{\ell}^0$ and 
$\widetilde{\mathcal{V}_\ell} \subset \widehat{V}_\ell^1$ if $\widetilde{\mathcal{M}}_{\ell} = 
\widetilde{\mathcal{M}}_{\ell}^1$.
\end{lemma}
\begin{proof}
We first note that if $T \in \widetilde{\TT}_{\ell} \setminus \widetilde{\TT}_{\ell-1}$, then 
$h_T \simeq \widehat{h}_{\ell}$. If 
$T \notin \TB_{1,\ell}$  for the first set  in the definition of the finest common coarsening (\ref{eq:def-fcc}),
then $T \in \widehat{\TT}_\ell$ and $h_T \simeq \widehat{h}_{\ell}$ follows since the mesh $\widehat{\TT}_\ell$ 
is quasi-uniform. 
Now, let $T \in \TB_{1,\ell}$, which implies $T \in \TT$, and that $T$ is a proper 
superset of an element $\widehat T_\ell \in \widehat{\TT}_\ell$, i.e., $h_T \geq \widehat{h}_{\ell}$. 
Since $\TT$ and $\widehat \TT_{\ell-1}$ are NVB-refinements of the same mesh, we either have 
$T \subset \widehat T_{\ell-1}$, $T = \widehat T_{\ell-1}$ or 
$T\supset \widehat T_{\ell-1}$ for some element $\widehat T_{\ell-1} \in \widehat{\TT}_{\ell-1}$. For the first two cases, we have $h_T \lesssim \widehat{h}_{\ell-1} \simeq 2\widehat{h}_{\ell}$,
which gives $h_T \simeq \widehat{h}_{\ell}$.
The third case $T\supset \widehat T_{\ell-1}$ implies that $T \in \TB_{1,\ell-1}$ and therefore $T \in 
\widetilde{\TT}_{\ell-1}$, which contradicts the assumption 
$T \in \widetilde{\TT}_{\ell} \setminus \widetilde{\TT}_{\ell-1}$.

This immediately proves the case $\widetilde{\mathcal{M}}_{\ell} = \widetilde{\mathcal{M}}_{\ell}^0$, since 
new points in $\widetilde{\mathcal{M}}_{\ell}^0$ (barycenters) correspond to new elements in 
$\widetilde{\TT}_{\ell} \setminus \widetilde{\TT}_{\ell-1}$.

For the case $\widetilde{\mathcal{M}}_{\ell} = \widetilde{\mathcal{M}}_{\ell}^1$,
let $z \in \widetilde{\mathcal{M}}_{\ell}$. By definition, this implies that there exists (at least) one element 
$T_z \in \omega_{\ell}(z)$ with $T_z \in \widetilde{\TT}_{\ell} \setminus \widetilde{\TT}_{\ell-1}$. The 
previous discussion gives $h_{T_z}\simeq \widehat{h}_{\ell}$. By shape-regularity this gives that 
$h_\ell(z) = \max_{T\in \omega_\ell(z)}h_T \simeq \widehat{h}_{\ell}$.
\end{proof}

With the inverse estimate of the previous subsection we now prove a 
strengthened Cauchy-Schwarz inequality.

\begin{lemma}\label{lem:strengthenedCS}
(Strengthened Cauchy-Schwarz inequality for the mesh hierarchy $(\widetilde\TT_\ell)$:) 
Let 
$u_{\ell} \in \widetilde{\mathcal{V}}_\ell$ for $\ell =0,1,...,L$. Then, we have 
\begin{align*}
a( u_m,u_k) \le \mathcal{E}_{km}
\norm{u _m}_{\widetilde H^s(\Omega)}\norm{\widehat h_k^{-s} u_k}_{L^2(\Omega)} \qquad  0 \le m \le  k \le L,
\end{align*}
with $\mathcal{E}_{km} = C_{\rm CS}  {( \widehat{h}_k / \widehat{h}_m)}^ {s-\beta}$.
Here, $\beta$ is given as $\beta =  \begin{cases}
  0& \text{for $0 < s < \frac 1 4$}\\
   \frac 3 2 s -\frac 1 4& \text{for $\frac 1 4 \leq s < \frac 1 2$}\\
 \end{cases}$ for the piecewise constant case and 
$\beta = \max\{s-1/2,0\}$ for the piecewise linear case.
Moreover,  the appearing constant $C_{\rm CS}>0$ depends only on $\Omega, d, s$ and the initial mesh  $\TT_0$.
\end{lemma} 
\begin{proof}
	We define  a   modified mesh size function $\widetilde h_m^s$ as 
	$\widetilde h_m^s := h_m^{s-\beta} w_m^\beta$  with the weight function $w_m$ 
	defined such that the inverse estimates of 
	\eqref{eq:investS11}, \eqref{eq:investS12} or Lemma~\ref{lem:invest} (either for the piecewise linears  
	or the piecewise constants) hold. Moreover,
	we note that this choice of $\beta$ fulfills 
	the assumptions of Lemma~\ref{lem:invest} as well as $\beta < s$.
Therefore, the classical Cauchy-Schwarz inequality implies
\begin{align}\label{CS 18}
\nonumber a( u_m,u_k)=&\skp{(-\Delta)^su_m, u_k }_{L^2(\Omega)}=  \skp{\widetilde{h}_m ^s(-\Delta)^su_m, \widetilde{h}_m ^{-s} u_k }_{L^2(\Omega)}\\
& \le \left\|\widetilde h_m^s(-\Delta)^su_{m}\right\| _ {L^2 (\Omega)}
 \left\| \widetilde{h}_m ^{-s} u_k\right\|  _ {L^2(\Omega)}.
\end{align}
A scaling argument as in \cite[Lem.~3.2.]{faustmann2019quasi} yields 
\begin{align*}
\left\| w _k ^{-\beta} u_k\right\|  _ {L^2(T)}
\lesssim h_k^{s-\beta}(T)  \norm{u_k}_ {H^s(T)}+ h_ k ^{-\beta}(T) \norm{u_k}_{L^2 (T)}.
 \end{align*}
Together with $w_k \le w_m$, since $\widetilde{\mathcal{T}}_k$ is a refinement of  $\widetilde{\mathcal{T}}_m$,
and $h_m(T):=h_m|_T \geq \widehat h_m$ this gives
\begin{align*}
\nonumber  \left\| \widetilde{h}_m ^{-s} u_k\right\|  _ {L^2(T)}& \lesssim h_m ^{\beta -s}(T)\left\| w _k ^{-\beta} u_k\right\|  _ {L^2(T)}
\lesssim
 h_m ^{\beta -s}(T) \left( h_k^{s-\beta}(T)  \norm{u_k}   _ {H^s(T)}+ h_ k ^{-\beta}(T) \norm{u_k}_{L^2 (T)}\right) \\
 & \nonumber \lesssim  \widehat{h}_m ^{\beta -s} h_k^{s-\beta}(T)  \norm{ u_k}  _ { H^s(T)}+ \widehat{h}_m ^{\beta -s}h_ k ^{-\beta}(T) \norm{u_k}_{L^2 (T)}\\
  & \nonumber 
 \lesssim 
 \widehat{h}_m ^{\beta -s} h_k^{-\beta}(T)  \norm{u_k} _ {L^2(T)}+ {( \widehat{h}_k / \widehat{h}_m)}^ {s-\beta} \norm{\widehat h _k ^{-s}u_k}_{L^2 (T)}\\ \nonumber
 &\lesssim {( \widehat{h}_k / \widehat{h}_m)}^ {s-\beta}  \norm{\widehat h_k^ {-s} u_k} _ {L^2(T)}.
\end{align*}
Summation over all the elements of $\widetilde{ \mathcal{T}}_m$ gives
 \begin{align}\label{CS 19}
 \left\| \widetilde{h}_m ^{-s} u_k\right\|  _ {L^2(\Omega)} \lesssim 
 {( \widehat{h}_k / \widehat{h}_m)}^ {s-\beta}  \norm{\widehat h_k^ {-s} u_k} _ {L^2(\Omega)}.
 \end{align}
Combining \eqref{CS 18} and \eqref{CS 19} with the inverse estimate
\begin{align*}
 \left\|\widetilde h_m^s(-\Delta)^su_{m}\right\| _ {L^2 (\Omega)} \lesssim \norm{u_m}_{\widetilde{H}^s(\Omega)}
\end{align*}
of \eqref{eq:investS11}, \eqref{eq:investS12} or Lemma~\ref{lem:invest} 
proves the strengthened Cauchy-Schwarz inequality.
\end{proof}

\begin{remark}
 Since  ${( \widehat{h}_k / \widehat{h}_m)}^ {s-\beta} = 2^{-(k-m)(s-\beta)}$ 
for $0 \le m \le  k \le L$,  we get -- following 
 \cite{TosWid05} -- that the symmetric 
 matrix $\mathcal{E}$ with upper triangular part
 	given by $\mathcal{E}_{km} =C_{\rm CS}{( \widehat{h}_k / \widehat{h}_m)}^ {s-\beta}$ satisfies 
 $\rho(\mathcal{E}) < C_{\rm spr}$, with a constant depending only on 
$\Omega$, $d$, $s$, and the 
 initial triangulation $\TT_0$. 
\eremk
\end{remark}

%
%
\begin{lemma} \label{lem:localstab}
(Local stability).
For all $u_\ell \in \widetilde{\mathcal{V}}_\ell$, we have 
\begin{align*}
\norm{u_{\ell}}_{\widetilde{H}^s(\Omega)}^2 \leq
C_{\rm loc}\; \widetilde a_\ell(u_\ell,u_\ell)
\end{align*}
with a constant $C_{\rm loc}>0$ depending only on $\Omega, d, s$, and the initial triangulation $\TT_0$.
\end{lemma}

\begin{proof}
Since $u _\ell \in \widetilde{\mathcal{V}}_\ell$, we have 
$
 u_\ell=  \sum_{z \in\widetilde{\mathcal{M}}_\ell} u _ \ell (z) \widetilde\varphi ^{\ell}_z
$. With an inverse estimate - which is allowed since due to 
Lemma~\ref{lem:Ml} $u_\ell$ only lives on a quasi-uniform submesh - 
we can estimate using that the number of overlapping basis functions $\widetilde\varphi ^{\ell}_z$ is bounded by a constant 
depending only on the $\gamma$-shape regularity of the initial triangulation
\begin{align*}
\norm{u_ \ell}_{\widetilde H^s(\Omega)}^2 &\lesssim \norm{\widehat h_\ell^{-s} u_ \ell}_{L^2(\Omega)}^2 =
\widehat h_\ell^{-2s}\norm{\sum_{z \in\widetilde{\mathcal{M}}_\ell} u_\ell(z)\widetilde\varphi^{\ell}_z}_{L^2(\Omega)}^2
\lesssim \widehat h_\ell^{-2s}\sum_{z \in\widetilde{\mathcal{M}}_\ell}\abs{u_\ell(z)}^2 \norm{\widetilde\varphi^{\ell}_z}_{L^2(\Omega)}^2.
\end{align*}
By definition of $\widetilde a_\ell(\cdot,\cdot)$, this 
finishes the proof.
\end{proof}

For $0 \le k \le \ell \le L$, let $\mathcal{E}$ be a symmetric matrix with upper triangular part
		given by $\mathcal{E}_{\ell k} =C_{\rm CS}{( \widehat{h}_\ell / \widehat{h}_k)}^ {s-\beta}$. 
Then, the assumptions of Proposition~\ref{prop:Equiv} follow directly from Lemma~\ref{lem:stabledec} (lower bound)
and Lemma~\ref{lem:strengthenedCS} together with Lemma~\ref{lem:localstab} (upper bound) by writing $u = \sum_{k} u_k$ and 
\begin{align*}
 a(u,u) &= \sum_{k,\ell=1}^L a(u_k,u_\ell) \le 2 \sum_{\ell=1}^L \sum_{k=1}^\ell  a(u_k,u_\ell) 
\stackrel{\text{Lemma~\ref{lem:strengthenedCS}}}{\leq}
 2\sum_{\ell=1}^L \sum_{k=1}^\ell  \mathcal{E}_{\ell k}\sqrt{a(u_k,u_k)\;\widetilde{a}_\ell(u_\ell,u_\ell)} \\
 &\stackrel{\text{Lemma~\ref{lem:localstab}}}{\leq}2 C_{\rm loc}^{1/2} \sum_{\ell=1}^L \sum_{k=1}^\ell  \mathcal{E}_{\ell k}\sqrt{\widetilde{a}_k(u_k,u_k)\;\widetilde{a}_\ell(u_\ell,u_\ell)} 
\leq 2C_{\rm loc}^{1/2} \rho(\mathcal{E}) \sum_{\ell=0}^L \widetilde{a}_\ell(u_\ell,u_\ell),
\end{align*}
and the appearing constants are independent of $L$.

\begin{remark} (Stable decomposition and strengthened Cauchy-Schwarz inequality of mesh hierarchy $(\TT_\ell)_\ell$ 
generated by an adaptive algorithm):
The existence of a stable decomposition and consequently the lower bound in Proposition~\ref{prop:spectralEquiv}
follows essentially verbatim as in \cite[Sec.~4.5]{feischl2017optimal}, where instead of 
Corollary~\ref{cor:lowerbound} an $L^2$-orthogonal projection onto a uniform mesh is used. 

Analyzing the proof of Lemma~\ref{lem:strengthenedCS}, we observe that the choice of mesh hierarchy is 
not crucial for the arguments, one only needs an inverse estimate and a Poincar\'e-type inequality.
Both hold for the case of the decomposition into one dimensional spaces $V_z^\ell$ 
instead of $\widetilde{\mathcal{V}}^\ell$ as well, and, therefore, we directly obtain a strengthened Cauchy-Schwarz inequality 
for $(\TT_\ell)_\ell$ as well. The algebraic arguments of \cite[Sec.~4.6]{feischl2017optimal} then give the upper
bound for Proposition~\ref{prop:spectralEquiv}. \eremk\bigskip
\end{remark}

 \begin{remark}
 In the same way as in \cite{feischl2017optimal}, 
 it is possible to define a \emph{global multilevel diagonal preconditioner} by taking the whole diagonal of the 
 matrix $\mathbf{A}^\ell$ instead of only the diagonal corresponding to the nodes in $\mathcal{M}_\ell$.
 However, compared to the local multilevel diagonal preconditioner, the preconditioner is not optimal in the sense that
 the condition number of the preconditioned system grows (theoretically) by a logarithmic factor of $N_L$.
 We refer to \cite{feischl2017optimal} for numerical observations of the sharpness of this bound for the hyper-singular
 integral operator in the BEM, which essentially corresponds to the case $s=1/2$ here.
\eremk
 \end{remark}
\subsection{Numerical example}
\begin{figure}[ht]
\includegraphics[width=0.4\textwidth]{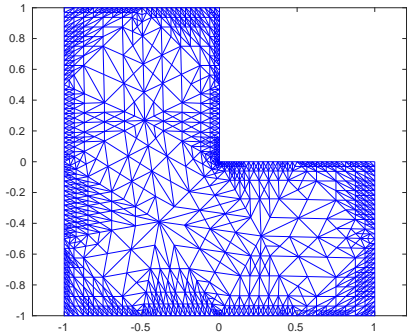}
\centering
\caption{Adaptively generated NVB mesh on the L-shaped domain.}
\label{fig:Lshape}
\end{figure}
We consider the L-shaped domain $\Omega = (-1,1)^2 \backslash [0,1]^2$ as depicted in 
Figure~\ref{fig:Lshape} and discretize \eqref{eq:modelproblem} by piecewise linear functions 
in $S^{1,1}_0(\TT_\ell)$ on adaptively generated NVB meshes $\TT_\ell$ that are 
generated by the adaptive algorithm proposed in \cite{faustmann2019quasi}. 
This adaptive algorithm is steered by local error indicators given by 
\begin{align*}
 \eta_\ell = \left(\sum_{T\in\TT_\ell} 
 \norm{\widetilde h_\ell^{s}\big(f-(-\Delta)^s u_\ell\big)}_{L^2(T)}^2 \right)^{1/2}\!, 
 \quad \text{with} \quad
\widetilde h_\ell^{s} :=
 \begin{cases}
   h_\ell^{s}\quad& \text{for $0 < s \le 1/2$},\\
   h_\ell^{1/2}w_\ell^{s-1/2} & \text{for $1/2 < s < 1$,}
 \end{cases}
\end{align*}
where $u_\ell$ is the solution of \eqref{eq:GalerkinDiscretization}.
We note that by \cite[Theorem~2.3]{faustmann2019quasi} theses indicators are reliable and for $s<1/2$
efficient in some weak sense. Moreover, \cite[Theorem~2.6]{faustmann2019quasi} proves optimal convergence
rates for the adaptive algorithm based on these estimators.
Our implementation of the classical {\tt SOLVE-ESTIMATE-MARK-REFINE} adaptive algorithm uses the 
{\tt MATLAB} code from \cite{ABB17} for the module {\tt SOLVE} and adapted the {\tt MATLAB} code for the 
local multilevel preconditioner from \cite{feischl2017optimal} to our model problem.
Figure~\ref{fig:precLshape} gives the estimated condition numbers for the Galerkin matrix 
$\mathbf{A}^L$ and the preconditioned matrix $\widetilde{\mathbf{P}}_{AS}^{L}$, where the
condition number has been estimated using power iteration and inverse power iteration 
(with random initial vectors) to 
compute approximations to the smallest and largest eigenvalues.

\begin{figure}[ht]
\begin{minipage}{.50\linewidth}
\includegraphics{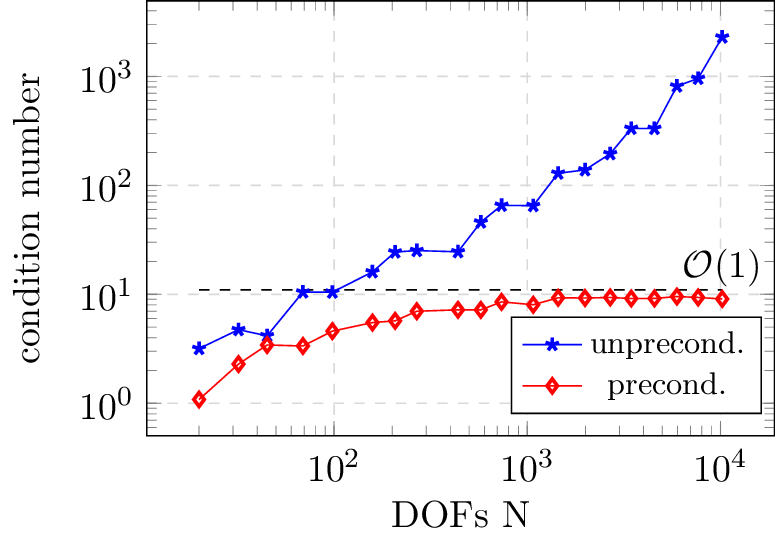}
\end{minipage}
\begin{minipage}{.49\linewidth}
\includegraphics{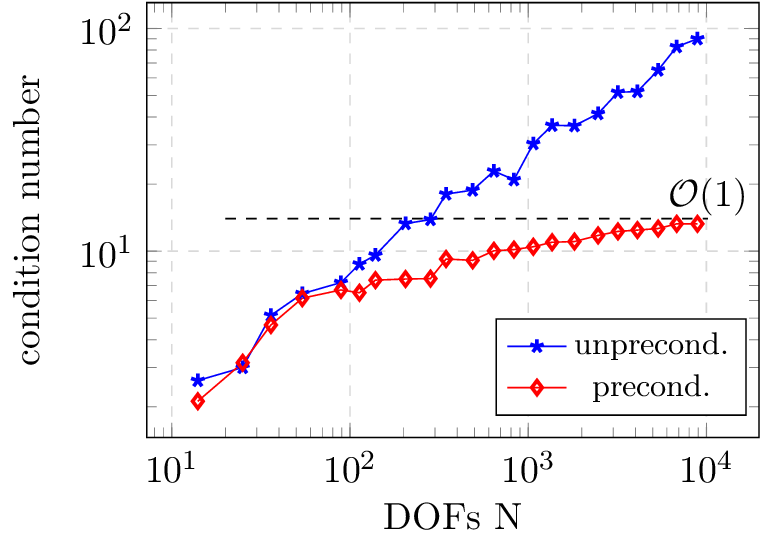}
\end{minipage}
\centering
\caption{Estimated condition numbers for the unpreconditioned matrix $\mathbf{A}^L$ and 
the preconditioned matrix $\mathbf{P}_{AS}^{L}$
left: $s=0.25$, 
right: $s=0.75$.}
\label{fig:precLshape}
\end{figure}

We observe that, as expected,  the condition number of the preconditioned system grows with the problem size, 
whereas the preconditioner leads to uniformly bounded condition numbers for the preconditioned system.

As the preconditioner is structurally similar to the one used in \cite{feischl2017optimal} for the hypersingular 
integral equation, we refer to the numerical results there for the confirmation that the preconditioner can also 
be realized efficiently. 

\bibliographystyle{amsalpha}
\bibliography{scott-zhang}

\end{document}